\documentclass[12pt,a4paper]{amsart}
\usepackage[english]{babel}
\usepackage{amsmath}
\usepackage{amsfonts}
\usepackage{amssymb}
\usepackage{amsthm}
\usepackage{graphicx}
\usepackage{enumitem}
\usepackage{stmaryrd}
\usepackage{hyperref}
\usepackage{bbm}
\usepackage{xcolor,mathtools}
\usepackage{mathrsfs}
\usepackage{esint}
\usepackage{cases}
\usepackage[numbers,sort]{natbib}

\theoremstyle{definition}
\newtheorem{defi}{Definition}[section]

\newtheorem{rem}[defi]{Remark}
\newtheorem*{notations}{Notation}
\newtheorem{construction}[defi]{Construction}
\theoremstyle{plain}
\newtheorem{teo}[defi]{Theorem}
\newtheorem{lem}[defi]{Lemma}
\newtheorem{pro}[defi]{Proposition}

\usepackage[left=2cm,right=2cm,top=2cm,bottom=2cm]{geometry}
\newcommand{\N}{\mathbb{N}}
\newcommand{\vp}{\varphi}
\newcommand{\R}{\mathbb{R}}
\newcommand{\pp}{\partial}
\newcommand{\ve}{\varepsilon}
\newcommand{\se}{\subseteq}

\newcommand{\ceq}{\coloneqq}

\newcommand{\res} {\mathop{\hbox{\vrule height 7pt width .5pt depth 0pt \vrule height .5pt width 6pt depth 0pt}}\nolimits}

\newcommand{\leb}{\mathcal{L}}
\newcommand{\sbv}{\operatorname{SBV}}
\newcommand{\bv}{\operatorname{BV}}
\newcommand{\J}{\mathcal{J}}
\newcommand{\hau}{\mathcal{H}}
\renewcommand{\div}{\operatorname{div}}
\newcommand{\nl}{\left\|}
\newcommand{\nr}{\right\|}
\newcommand{\spt}{\operatorname{spt}}
\newcommand{\eps}{{\ve^{\ell,s}_k}}
\newcommand{\lsk}{^{\ell,s}_k}

\newcommand{\loc}{\operatorname{loc}}
\newcommand{\ap}{{\operatorname{ap}}}
\newcommand{\lski}{^{k,\ell,s}_i}
\newcommand{\jf}{\mathfrak{j}}
\newcommand{\ch}{{\operatorname{ch}}}


\title[SBV functions in CC spaces]{SBV functions in Carnot-Carathéodory spaces}


\begin{document}

\author[M. Di Marco]{Marco Di Marco}
\address[M. Di Marco]{Dipartimento di Matematica ``T. Levi-Civita'', Università di Padova, via Trieste 63, 35121 Padova, Italy.}
\email{marco.dimarco@phd.unipd.it}

\author[S. Don]{Sebastiano Don}
\address[S. Don]{Dipartimento di Ingegneria Civile, Architettura, Territorio, Ambiente e di Matematica (DICATAM), via Branze, 38, 25123 Brescia, Italy}
\email{sebastiano.don@unibs.it}

\author[D. Vittone]{Davide Vittone}
\address[D. Vittone]{Dipartimento di Matematica ``T. Levi-Civita'', Università di Padova, via Trieste 63, 35121 Padova, Italy.}
\email{davide.vittone@unipd.it}

\begin{abstract}
We introduce the space SBV$_X$ of special functions with bounded $X$-variation in Carnot-Carathéodory spaces and study its main properties. Our main outcome is an approximation result, with respect to  the BV$_X$ topology, for SBV$_X$ functions.
\end{abstract}

\maketitle

\section{Introduction}

Functions with Bounded Variation (BV), and in particular their subclass of {\em special} functions with Bounded Variation (SBV), provide a natural framework for studying problems involving discontinuities, such as image processing, signal analysis, and variational problems. 
Over recent years, a considerable effort was put into the study  of BV functions in Carnot-Carathéodory (CC) spaces. The aim of this paper is to contribute to this area of research by introducing the space $\sbv_X$ of special functions with bounded $X$-variation and studying its properties. In particular, we extend  to the setting of CC spaces  the following  approximation result for classical SBV functions proved by G.~De Philippis, N.~Fusco and A.~Pratelli in~\cite{dpfp}. 

\begin{teo}[{\cite[Theorem A]{dpfp}}]\label{teo_introclassic}
    Let $\Omega \subset \R^n$ be an open set and let $u \in \sbv(\Omega)$. Then, there exists a sequence of functions $u_k \in \sbv(\Omega)$ and of compact $C^1$-manifolds $M_k \subset \subset \Omega$ such that $\J_{u_k} \subseteq M_k \cap \J_u$, $\hau^{n-1}(\overline{\J_{u_k}} \setminus \J_{u_k})=0$ and
    \[
    \nl u_k-u \nr_{\bv(\Omega)}\xrightarrow{k \to +\infty}0, \qquad u_k \in C^\infty (\Omega \setminus \overline{\J_{u_k}}).
    \]
\end{teo}

Recall that smooth functions are not dense in BV with respect to the BV topology\footnote{Smooth functions are dense in BV only with respect to the so-called strict topology in BV, see e.g.~\cite[Theorem~3.9]{afp}.}, as their closure in BV is the Sobolev space $W^{1,1}$, i.e., BV functions whose derivatives admit no singular part (not even ``nice'' jumps) with respect to the Lebesgue measure.  In this sense, Theorem~\ref{teo_introclassic} provides a class of ``nice'' (although, clearly, not smooth) BV functions that are dense in SBV with respect to the BV topology. As explained in~\cite{dpfp}, this result is sharp and, besides being interesting {\em per se}, it led to the proof of a conjecture by L.~Ambrosio, J.~Bourgain, H.~Brezis and A.~Figalli~\cite{abbf16} (see also~\cite{fms16}) about a formula for a BMO-seminorm (defined as an isotropic version of the BMO-norm introduced in \cite{bbm}) for SBV functions. Before stating our main result we need to briefly introduce the notion of special function of bounded variation in CC spaces. 
A {\em Carnot-Carathéodory space} (see Definition \ref{def_cc}) is the space $\R^n$ endowed with a distance arising from a collection $X=(X_1,\dots,X_m)$ of smooth and linearly independent vector fields satisfying the H\"ormander condition. In this paper, we will deal with {\em equiregular} CC spaces, where a homogeneous dimension $Q$, usually larger than the topological dimension $n$, can be defined. The space $\bv_X(\Omega)$~\cite{cdg,fssc} of functions with bounded $X$-variation (see Definition \ref{def_bvx}) consists of those functions $u$ on an open set $\Omega\subset\R^n$ whose derivatives $X_1u, \dots, X_mu$ in the sense of distributions are represented by a vector-valued measure $D_Xu$ with finite total variation $|D_Xu|$. The space $\bv_X$ has been the subject of intensive studies: see \cite{bu95,dgn98,fgw94,fssc01,fssc03,gn96,garofalonhieu,CapGarAhlfors,DanGarNhi,Selby} and the more recent \cite{agm15,am03,as10,bmp12,cm20,dmv19,magnani02,marchi14,sy03,vittone2012,dontesi,dv,dv19}.

In the classical Euclidean setting the space of SBV functions, first introduced in \cite{degiorgiambrosio}, naturally arises in the study of free discontinuity problems. 
The first contribution of this paper is the introduction of special functions with Bounded Variation in   CC spaces ($\sbv_X$ functions). Recall (\cite{dv}) that, if $u \in \bv_X(\Omega)$, then $D_Xu$ can be decomposed as
\[
D_Xu=D^{\ap}_Xu\: \leb^n+D^s_Xu=D^{\ap}_Xu\: \leb^n+D^j_Xu+D^c_Xu,
\]
where $D^\ap_Xu$ is the approximate $X$-gradient of $u$, $\leb^n$ is the usual Lebesgue measure, $D_X^su$ is the singular part of $D_Xu$, $D^j_Xu$ is the jump part of $D_X^ju$, and $D_X^cu$ is the Cantor part of $D_Xu$. See Section~\ref{sec:preliminari} for precise definitions. 

\begin{defi}\label{def_introsbvx}
 Let $\Omega\subset\R^n$ be an open subset of an equiregular Carnot-Carathéodory space $(\R^n,X)$ and let $u \in \bv_X(\Omega)$. We say that $u$ is a \emph{special function of bounded $X$-variation}, and we write $u \in \sbv_X(\Omega)$, if
\begin{enumerate}
    \item[(i)] $D_X^c u=0$, and
    \item[(ii)] the jump set $\J_u$ of $u$ is a countably $X$-rectifiable set.
\end{enumerate}
\end{defi}

A set is said to be countably $X$-rectifiable (see Definition~\ref{def_Xrectifiable}) if it can be covered, up to a set which is negligible with respect to the Hausdorff measure $\hau^{Q-1}$, by a countable family of $C^1_X$-hypersurfaces (Definition \ref{def_ipersup}), that provide the  intrinsic counterpart in CC spaces of classical $C^1$-hypersurfaces. Recall that, for classical BV functions, the jump set is always countably rectifiable; on the contrary, in CC spaces this -- i.e., the validity of condition (ii) above for every $\bv_X$ function $u$ -- is an important open problem. Let us however recall that, if the CC space satisfies the so-called {\em property $\mathcal{R}$} (``rectifiability'', see Definition \ref{def_pror}), then condition (ii) in Definition \ref{def_introsbvx} is automatically satisfied for every $u \in \bv_X$; see~\cite[Theorem 1.5]{dv}. 
There is a multitude of examples of CC spaces which satisfy  property $\mathcal{R}$, such as Heisenberg groups, step 2 Carnot groups and Carnot groups of type $\star$, see \cite[Theorem 4.3]{dv}. In this paper we tried to work in the widest possible generality, hence the extra requirement (ii) in Definition \ref{def_introsbvx}. 
For this reason, let us also stress that our definition might be a priori different from the definition of SBV function in metric measure spaces. We refer to \cite{AmbrosioMirandaPallaraSBV} for a general overview of SBV functions in metric measure spaces and to \cite{Lahti2020} for an approximation result for BV functions via SBV functions in this context.

In Section~\ref{sec_sbvx} we study several properties of $\sbv_X$ (or locally $\sbv_X$) functions: we collect the main ones in the following theorem, which summarizes (some of) the results stated in Proposition~\ref{prop_defequivsbvx}, Theorem~\ref{teo_sbvxchiuso}, Lemma~\ref{lem_bvbvx}, Theorem~\ref{teo_albertisbvx} and Theorem~\ref{teo_produzionedisbvx}.

\begin{teo}\label{teo_proprietavariesvbx}
    Let $\Omega\subset\R^n$ be an open subset of an equiregular Carnot-Carathéodory space $(\R^n,X)$; then, the following statements hold:
    \begin{enumerate}
        \item[(i)] $u\in \sbv_{X,\loc}(\Omega)$ if and only if $D_X^su = f\:\nu_R\:\hau^{Q-1}\res R$ for some countably $X$-rectifiable set $R\subset\Omega$ with horizontal normal $\nu_R$ and some   $f\in L^1_{\loc}(R,\hau^{Q-1})$;
        \item[(ii)] $\sbv_X(\Omega)$ is a closed subspace of $\bv_X(\Omega)$;
        \item[(iii)] the space $\sbv_{\loc}(\Omega)$ of special function of (Euclidean) locally bounded variation is contained in $ \sbv_{X,\loc}(\Omega)$;
        \item[(iv)] for every $w \in L^1_{\loc}(\Omega;\R^m)$ there exists $u \in \sbv_{X,\loc}(\Omega)$ such that $D^{\ap}_Xu=w$ a.e. in $\Omega$;
        \item[(v)] for every  countably $X$-rectifiable set $R \se \Omega$  oriented by $\nu_R$, every $  \theta \in L^1(\hau^{Q-1} \res R)$ and every $\delta>0$ there exists $u \in \sbv_{X}(\Omega)$ such that
        \[
        D_X^ju \equiv \theta\,  \nu_R\, \hau^{Q-1} \res R, \quad \nl u \nr_{L^1(\Omega)}<\delta, \quad \text{ and }\quad |D_Xu| (\Omega) \leq (2+\delta)\nl \theta  \nr_{L^1(\hau^{Q-1} \res R)}.
        \]
    \end{enumerate}
\end{teo}

Statements $(iii),\ (iv)$ and $(v)$, in particular, provide  meaningful subclasses or  examples of special functions of bounded $X$-variation which, in particular, turn out to form a quite large and interesting space.

We can now to state our main result.

\begin{teo}\label{teo_introapprox}
Let $\Omega$ be an open subset of an equiregular Carnot-Carathéodory space $(\R^n,X)$ and let $u \in \sbv_X(\Omega)$. Then, there exists a sequence of functions $(u_k)_{k \in \N} \subset \sbv_X(\Omega)$ and of $C^1_X$-hypersurfaces $(M_k)_{k \in \N} \subset \Omega$ such that, for every $k \in \N$, $\J_{u_k} \se M_k \cap \J_u$, $\J_{u_k}$ is compact, and
\[
\nl u-u_k \nr_{\bv_X(\Omega)}\xrightarrow{k \to +\infty}0, \qquad u_k \in C^\infty(\Omega \setminus \J_{u_k}).
\]
\end{teo}

Our proof of Theorem~\ref{teo_introapprox} differs from the one of Theorem \ref{teo_introclassic} in that, rather than using  mollifications with variable kernel as in~\cite{dpfp}, we exploit a partition-of-the-unity argument (reminiscent of~\cite{ag78,fssc,vittone2012,vittone22,stokes}) that allows to approximate $u$ out of a fixed compact set  $C_k$. {\em A posteriori,} the set $C_k$ coincides with the jump set $\J_{u_k}$, which in particular turns out to be compact itself, thus providing a slight improvement in Theorem~\ref{teo_introclassic}.

We believe that Theorem~\ref{teo_introapprox} will play a role in a possible, future BMO-type characterization of $\bv_X$ functions {\em à la} Ambrosio-Bourgain-Brezis-Figalli, \cite{abbf16}: this, in fact, was one of the original motivations of our work, and will be the subject of further investigations. \medskip

{\bf Acknowledgments.} M.~D.~M. and D.~V. are supported by University of Padova and  GNAMPA of INdAM. D.~V. is also supported by INdAM project {\em VAC\&GMT} and by PRIN 2022PJ9EFL project {\em Geometric Measure Theory: Structure of Singular Measures, Regularity Theory and Applications in the Calculus of Variations} funded by the European Union - Next Generation EU, Mission 4, component 2 - CUP:E53D23005860006. The authors warmly thank the anonymous referee for her/his careful reading
and the precious suggestions.

\section{Definitions and preliminary results}\label{sec:preliminari}

\begin{defi}\label{def_cc}
Let $1\leq m \leq n$ be integers and let $X=(X_1,\dots,X_m)$ be a $m$-tuple of smooth and linearly independent vector fields on $\R^n$. We say that an absolutely continuous curve $\gamma\colon[0,T] \to \R^n$ is an \emph{$X$-subunit path} joining $p$ and $q$ if $\gamma(0)=p$, $\gamma(T)=q$ and there exist $h_1,\dots h_m \in L^\infty([0,T])$ such that $\sum_{j=1}^m h_j^2 \leq 1$ and 
\[
\gamma'(t)=\sum_{j=1}^m h_j(t)X_j(\gamma(t)), \quad \text{for a.e.\ $t\in [0,T]$.}
\]
For every $p,q \in \R^n$ we  define  
\[
d(p,q) \ceq \inf \lbrace T>0: \text{ there exists an $X$-subunit path $\gamma$ joining $p$ and $q$}\rbrace,
\]
where we agree that $\inf \emptyset \ceq +\infty$.

By the Chow–Rashevskii Theorem (see for instance \cite[Subsection 3.2.1]{abb}), if for every $p \in \R^n$ the linear span of all iterated commutators of the vector fields $X_1,\dots, X_m$ computed at $p$ has dimension $n$ (i.e. $X_1,\dots,X_m$ satisfy the \emph{H\"ormander condition}), then $d$ is a distance: the latter means that for every couple of points of $\R^n$ there always exists a $X$-subunit path joining them. In this case we say that $(\R^n,X)$ is a \emph{Carnot-Carathéodory space} of \emph{rank} $m$ and $d$ is the associated \emph{Carnot-Carathéodory distance}.         

For every $p \in \R^n$ and for every $i \in \N$ we denote by $\mathfrak{L}^i(p)$ the linear span of all the commutators of $X_1,\dots,X_m$ up to order $i$ computed at $p$. We say that a Carnot-Carathéodory space $(\R^n,X)$ is \emph{equiregular} if there exist natural numbers $n_0,n_1,\dots,n_s$ such that 
\[
0=n_0<n_1<\cdots<n_s=n \text{ and } \dim  \mathfrak{L}^i(p)=n_i, \quad \forall p \in \R^n, \forall i\in\{1,\dots, n\}.
\]
 The natural number $s$ is called \emph{step} of the Carnot-Carathéodory space. If $(\R^n,X)$ is equiregular, then the \emph{homogeneous dimension} is $Q \ceq \sum_{i=1}^si(n_i-n_{i-1})$.
\end{defi}

\begin{notations}
{\bf In the following, $\mathbf{(\R^n,X)}$ denotes an equiregular Carnot-Carathéodory space associated with the family $\mathbf{X=(X_1,\dots,X_m)}$.} We use $d$ to denote the Carnot-Carathéodory distance associated with $X$, $B(\cdot,\cdot)$ to denote the associated open balls, $\hau^k$ to denote the associated Hausdorff $k$-measure and $\mathcal S^k$ to denote the associated spherical Hausdorff $k$-measure; on the other hand we will denote by $d_E$ the usual Euclidean distance, by $B_E(\cdot,\cdot)$ the associated open balls, by $\hau^k_E$ the associated Hausdorff $k$-measure and by $\mathcal S^k_E$ the associated spherical Hausdorff $k$-measure. By $\Omega \se (\R^n,X)$ we denote a fixed open set and by $Q$ we denote the homogeneous dimension of $(\R^n,X)$.  Later we will also use the following notation: 
\begin{itemize}
\item for every $1 \leq i \leq m$ and $x \in \R^n$ we write
\[
X_i(x)=(a_{i,1}(x),\dots,a_{i,n}(x))
\]
where $a_{i,t} \in C^\infty(\R^n)$ for $1 \leq t \leq n$;
\item for every $1 \leq i \leq m$ and $x \in \R^n$ we write
\[
(\div X_i)(x) \ceq \sum_{t=1}^n \frac{\pp a_{i,t}}{\pp x_t}(x)
\]
\item for every $1 \leq i \leq m$, $\vp \in C^1(\Omega)$ and $x \in \Omega$ we write
\[
(X_i\vp)(x) \ceq \sum_{t=1}^n a_{i,t}(x)\frac{\pp \vp}{\pp x_t}(x);
\]
\item for every $1 \leq i \leq m$ we denote by $X_i^*$ the formal adjoint of $X_i$, i.e., for every $\vp \in C^1(\Omega)$, $x \in \Omega$ we write
\[
(X_i^*\vp)(x) \ceq \sum_{t=1}^n \frac{\pp (a_{i,t}\vp)}{\pp x_t}(x);
\]
\item given a Radon measure $\mu$, we use the notation
\[
\fint_A u d \mu \ceq \frac{1}{\mu(A)}\int_A u d \mu,
\]
to denote the average integral of a measurable function $u$ on a $\mu$-measurable set $A$ with $\mu(A)>0$;
\item we will consider a Riemannian metric, namely a smoothly varying family of scalar products $\langle \cdot, \cdot \rangle_p=\langle\cdot,\cdot \rangle\colon \mathbb R^n\times \mathbb R^n\to \R$ which makes the horizontal vectors $X_1,\dots,X_m$ orthonormal at every point $p\in \mathbb R^n$.
\end{itemize}
\end{notations}
\begin{defi}\label{def_bvx}
We say that $u \in L^1_{\loc}(\Omega)$ is a \emph{function of locally bounded $X$-variation}, and we write $u \in \bv_{X,\loc}(\Omega)$, if there exists a $\R^m$-valued Radon measure $D_X u=(D_{X_1}u,\dots D_{X_m}u)$ on $\Omega$ such that, for every open set $A \subset \subset \Omega$, for every $1 \leq i \leq m$ and for every $\vp \in C^1_c(A)$ one has 
\[
\int_A \vp d(D_{X_i}u)=-\int_A u X_i^*\vp d \leb^n.
 \]
Moreover, if $u \in L^1(\Omega)$ and $D_X u$ has bounded total variation $|D_X u|$, then we say that $u$ has \emph{bounded $X$-variation} and we write $u \in \bv_X(\Omega)$. 
\end{defi}
\begin{defi} 
For every $u \in \bv_X(\Omega)$ we define the norm
\[
\nl u \nr_{\bv_X(\Omega)} \ceq \nl u \nr_{L^1(\Omega)}+|D_Xu|(\Omega).
\]
The space $\bv_X(\Omega)$ equipped with the above norm is a Banach space.
\end{defi}
\begin{defi}
For every $u \in \bv_X(\Omega)$ we decompose
\[
D_Xu=D_X^au+D_X^su
\]
where $D_X^au$ denotes the \emph{absolutely continuous part} of $D_Xu$ (with respect to the usual Lebesgue measure $\leb^n$) and $D_X^su$ denotes the \emph{singular part} of $D_Xu$.
\end{defi}

\begin{defi}
    We say that a measurable set  $E \se \R^n$ has \emph{locally finite $X$-perimeter} (respectively, \emph{finite $X$-perimeter}) in $\Omega$ if its characteristic function $\chi_E$ belongs to $\bv_{X,\loc}(\Omega)$ (respectively, $\chi_E \in \bv_X(\Omega)$). 
    In such a case we define the \emph{$X$-perimeter measure} $P^X_E$ of $E$ as $P^X_E:=|D_X \chi_E|$.
\end{defi}

\begin{defi}\label{def_c1x}
Let $\Omega \se (\R^n,X)$ be an open set and $f\colon \Omega \to \R$. We say that $f \in C^1_X(\Omega)$ if $f$ is continuous and its \emph{horizontal gradient}  $Xf \ceq (X_1f,\dots,X_mf)$,   in the sense of distributions, is represented by a continuous function.
\end{defi}

\begin{defi}
    Let $u \in L^1_{\loc}(\Omega)$, $z \in \R$ and $p \in \Omega$. We say that $z$ is the \emph{approximate limit} of $u$ at $p$ if
    \[
    \lim_{r \to 0} \fint_{B(p,r)}|u-z|d\leb^n=0.
    \]
If the approximate limit of $u$ at $p$ exists, it is also unique (see \cite[Definition 2.19]{dv}). We hence denote by $u^\star(p)$ the approximate limit of $u$ at $p$ and by $\mathcal{S}_u$ the subset of points in $\Omega$ where $u$ does not admit an approximate limit. 
\end{defi}

\begin{defi}\label{def_approxdiff}
    Let $u \in L^1_{\loc}(\Omega)$ and $p \in \Omega \setminus \mathcal{S}_u$. We say that $u$ is \emph{approximately $X$-differentiable} at $p$ if there exist a neighbourhood $U\subset \Omega$ of $p$ and $f \in C^1_X(U)$ such that $f(p)=0$ and
    \[
    \lim_{r \to 0}\fint_{B(p,r)} \frac{|u-u^\star (p)-f|}{r}d \leb^n =0.
    \]
    The set of points in $\Omega$ where $u$ is approximately $X$-differentiable is denoted by $\mathcal{D}_u$.
The vector $Xf(p) \in \R^m$ is uniquely determined (see~\cite[Proposition 2.30]{dv}): we call it \emph{approximate $X$-gradient} of $u$ at $p$ and we denote it by $D^\ap_X u(p)$. Similarly, we write  $D^\ap_{X_i}u(p):=X_if(p)$ for every $i \in \lbrace 1,\dots, m \rbrace$.
\end{defi}

The next two results collect some of the ``fine'' properties of $\bv_X$ functions proved in~\cite{dv}.

\begin{teo}[{\cite[Theorem 1.1]{dv}}]  \label{teo_gradapprox}  
Let $u \in \bv_X(\Omega)$. Then $u$ is approximately $X$-differentiable at $\leb^n$-almost every point of $\Omega$. Moreover, the approximate $X$-gradient coincides $\leb^n$-almost everywhere with the density of $D^a_Xu$ with respect to $\leb^n$.
\end{teo}

\begin{teo}[{\cite[Theorem 1.3]{dv}}] \label{teo_prop3.76}
	There exists $\lambda\colon \R^n\to(0,+\infty)$ locally bounded away from 0 such that, for every open set $\Omega\subset\R^n$  and every $u\in \bv_X(\Omega)$
	\[
	|D_Xu|\geq \lambda|u^+-u^-| \hau^{Q-1}\res \mathcal J_u.
	\]
Moreover, for every Borel set $B\subseteq \Omega$ the following implications hold:
	\begin{align}
		&\hau^{Q-1}(B)=0\quad \Rightarrow\quad |D_Xu|(B)=0;\label{1implication}\\
&	\label{2implication}
		\hau^{Q-1}(B)<+\infty \text{ and } B\cap\mathcal{S}_u=\emptyset\quad \Rightarrow \quad |D_Xu|(B)=0.
	\end{align}
\end{teo}

We  now  spend a few words about intrinsically $C^1$ (or $C^1_X$) hypersurfaces and the notion of $X$-rectifiability.

\begin{defi}\label{def_ipersup}
We say that $S \se (\R^n,X)$ is a \emph{$C^1_X$-hypersurface} if for every $p \in S$ there exist $r>0$ and $f \in C^1_X(B(p,r))$ such that the following facts hold:
\begin{enumerate}
\item[(i)]$ S \cap B(p,r)=\lbrace q \in B(p,r):f(q)=0 \rbrace$,
\item[(ii)]$Xf\neq 0$ on $B(p,r)$.
\end{enumerate}
We define the \emph{horizontal normal} to $S$ at $p \in S$ as
\[
\nu_S(p) \ceq \frac{Xf(p)}{|Xf(p)|}.
\]
Notice that $\nu_S(p)$ is well defined up to a sign and, in particular, it does not depend on the choice of $f$, see \cite[Corollary 2.14]{dv}.
\end{defi}

\begin{defi}\label{def_Xrectifiable}
Let $S \se (\R^n,X)$. We say that $S$ is \emph{countably $X$-rectifiable} if there exists a family $\lbrace S_h: h \in \N \rbrace$ of $C^1_X$-hypersurfaces such that
\[
\hau^{Q-1}\left( S \setminus \bigcup_{h \in \N}S_h \right)=0.
\]
  Moreover, if $\hau^{Q-1}(S)<+\infty$, we say that $S$ is \emph{$X$-rectifiable}. We define the \emph{horizontal normal} of a countably $X$-rectifiable set $S$ at $p \in S$ as 
\[
\nu_S(p) \ceq \nu_{S_h}(p) \text{ if }p \in S_h \setminus \bigcup_{k<h}S_k.
\]
Notice that $\nu_S$ is well defined, up to a sign, $\hau^{Q-1}$-a.e., see \cite[Proposition 2.18]{dv}.
\end{defi}

The following Lemma provides an equivalent definition of $X$-rectifiability; although not difficult and probably quite known (see \cite[Lemma 2.4]{JNGVIMRN} for a proof  in Heisenberg groups),  we include a proof for the sake of completeness. 
\begin{lem}\label{lem_unicaipersup}
A set $R \se (\R^n,X)$ is  $X$-rectifiable if and only if, for every $\ve>0$, there exists a $C^1_X$-hypersurface $S_\ve \se (\R^n,X)$ such that $\hau^{Q-1}(S_\ve)<\infty$ and
\[
\hau^{Q-1}(R \setminus S_\ve)<\ve.
\]
\end{lem}
\begin{proof}
Since $R$ is $X$-rectifiable we can write
\[
R  \se S_0 \cup \bigcup_{i \in \N} S_i
\]
where $S_0$ is a $\hau^{Q-1}$-negligible set and, for every $i \in \N$, $S_i$ is a $C^1_X$-hypersurface. It is not restrictive to assume $\hau^{Q-1}(S_i)<\infty$ for every $i \in \N$. For every $\ve>0$ there exists a positive integer $M$ such that
    \[
    \hau^{Q-1}\left( R \setminus \bigcup_{i \leq M}S_i \right)<\frac{\ve}{2}.
    \]    
    We define the $C^1_X$-hypersurface $S_1' \ceq \lbrace p \in S_1: d(p,\pp S_1)>r_1 \rbrace$, 
    where $\partial S_1':=\overline{S_1'}\setminus S_1'$ and $r_1$ is chosen so that
    \[
    \hau^{Q-1}(R \cap \pp S_1')=0 \quad \text{ and } \quad \hau^{Q-1}((R \cap S_1) \setminus S_1') < \frac{\ve}{4}.
    \]
    Let us prove that such $r_1$ exists: for $r>0$ we define the set
    \[
    S_1'(r) \ceq \lbrace p \in S_1: d(p,\pp S_1)> r \rbrace.
    \]
    Since $\lbrace R \cap \pp S_1'(r): r>0 \rbrace$ is a family of uncountably many pairwise disjoint subsets of $R$ and $\hau^{Q-1}(R)<\infty$, then $\hau^{Q-1}(R \cap \pp S_1'(r))=0$ for arbitrarily small $r>0$. Moreover, since $R \cap S_1$ is the union of the nested sets $(R \cap S_1) \setminus S_1'(r)$, we have, by the continuity of measure, that
    \[
    \hau^{Q-1}((R \cap S_1)\setminus S_1'(r))\xrightarrow{r \to 0}0.
    \]
    Reasoning by induction, for every $i=2,\dots,M$ we can define the $C^1_X$-hypersurfaces
    \[
    S_i'=\left\lbrace p \in S_i \setminus \bigcup_{j<i}\overline{S'_j}:d\left(p,\pp \left(S_i \setminus \bigcup_{j<i}\overline{S_j'}\right)\right)>r_j \right\rbrace,
    \]
    where we used the fact that $S_i \setminus \bigcup_{j<i} \overline{S_j'}$ is a $C^1_X$-hypersurface and $r_i>0$ is chosen so that
    \[
    \hau^{Q-1}(R \cap \pp S_i')=0 \quad \text{ and } \quad \hau^{Q-1}\left( R \cap \left( S_i \setminus \bigcup_{j<i} \overline{S_j'}\right) \setminus S_i' \right)< \frac{\ve}{2^{i+2}}.
    \]
    Now consider $S_\ve \ceq \bigcup_{i=1}^M S_i'$, which is a $C^1_X$-hypersurface because it is union of finitely many $C^1_X$-hypersurface at positive distance from each other. Then
    \begin{align*}
        \hau^{Q-1}(R \setminus S_\ve) &\leq \hau^{Q-1}\left( R \setminus \bigcup_{i \leq M}S_i \right)+\hau^{Q-1}\left( R \cap \left( \bigcup_{i \leq M}S_i \right) \setminus \left( \bigcup_{j \leq M}S_j'\right) \right)\\
        &< \frac{\ve}{2}+\hau^{Q-1} \left( \bigcup_{i \leq M} \left( \left( R \cap S_i \right) \setminus \bigcup_{j \leq M}S_j' \right)\right)\\
        &\leq \frac{\ve}{2}+\hau^{Q-1}\left( \bigcup_{i \leq M}\left(R \cap \left( S_i \setminus \bigcup_{j \leq i}S_j' \right) \right) \right)\\
        &=\frac{\ve}{2}+\hau^{Q-1}\left( \bigcup_{i \leq M}\left( R \cap \left( S_i \setminus \bigcup_{j<i}\overline{S_j'}\right) \setminus S_i' \right) \right)\\
        &<\ve,
    \end{align*}
where we used that $\hau^{Q-1}(R \cap \pp S_j')=0$. 
Finally we observe  that
\[
\hau^{Q-1}(S_\ve) = \sum_{i \leq M}\hau^{Q-1}(S_i') \leq \sum_{i \leq M}\hau^{Q-1}(S_i)<\infty.
\]
This proves one implication, the other one is trivial.
\end{proof}

\begin{defi}\label{def_b+b-}
Fix $p \in (\R^n,X)$, $R>0$ and $\nu \in \mathbb{S}^{m-1}$. Let $f \in C^1_X(B(p,R))$ be such that $f(p)=0$ and $\frac{Xf(p)}{|Xf(p)|}=\nu$. For every $r \in (0,R)$ we set
\begin{align*}
&B^+_\nu(p,r) \ceq B(p,r) \cap \lbrace f>0\rbrace,\\
&B^-_\nu(p,r) \ceq B(p,r) \cap \lbrace f<0\rbrace.
\end{align*}
\end{defi}

\begin{defi}\label{def_approxjump}
Let $u \in L^1_{loc}(\Omega)$ and $p \in \Omega$. We say that $u$ has an \emph{approximate $X$-jump} at $p$ if there exist $u^+,u^- \in \R$ with $u^+ \neq u^-$ and $\nu \in \mathbb{S}^{m-1}$ such that
\begin{equation}\label{eq_xjump}
\lim_{r \to 0} \fint_{B^+_\nu(p,r)} |u-u^+|d\leb^n=\lim_{r \to 0} \fint_{B^-_\nu(p,r)} |u-u^-|d\leb^n=0.
\end{equation}
The \emph{jump set} $\J_u$ is defined as the set of points where $u$ has an approximate $X$-jump. Notice that condition \eqref{eq_xjump} does not depend on the choice of the function $f$ used to construct the sets  ${B^+_\nu(p,r)}$ and ${B^-_\nu(p,r)}$, see \cite[Proposition 2.26 and Remark 2.27]{dv}.

It is worth remarking that, by~\cite[Theorem 1.2 and Remark 2.25]{dv}, the Hausdorff measure $\hau^{Q-1}$ on the jump set is $\sigma$-finite.
\end{defi}

\begin{rem}\label{rem_proprietaR}
It was proved in~\cite[Theorem 1.5]{dv} that, if $(\R^n,X)$ satisfies the additional property $\mathcal{R}$ (see Definition \ref{def_pror}  below), then the jump set $\J_u$ is countably $X$-rectifiable. It is worth recalling that Heisenberg groups, Carnot groups of step $2$ and Carnot groups of type $\star$ all satisfy property $\mathcal{R}$, see~\cite[Theorem 4.3]{dv}.
\end{rem}

\begin{defi}\label{def_pror}
      We say that $(\R^n,X)$ satisfies the \emph{property $\mathcal{R}$} if for every open set $\Omega \se \R^n$ and every $E \se \R^n$ with locally finite $X$-perimeter in $\Omega$, the essential boundary $\partial^* E \cap \Omega$ is countably $X$-rectifiable. Let us recall for completeness that the \emph{essential boundary} of a measurable set $E$ is $\partial^*E \ceq \R^n \setminus (E_0 \cup E_1)$, where for $\lambda \in [0,1]$ we denote by $E_\lambda$ the set of points $p\in\R^n$ where $E$ has density $\lambda$, i.e., 
    \[
    \lim_{r \to 0}\frac{\leb^n(E \cap B(p,r))}{\leb^n(B(p,r))}=\lambda.
    \]
    \end{defi}

We are ready to introduce the jump part of the derivative $D_Xu$.

\begin{defi}
Let $u \in \bv_X(\Omega)$. We define the \emph{jump part} of $D_Xu$ as 
\[
D_X^ju \ceq D_X^s u \res \J_u
\]
and the \emph{Cantor part} of $D_Xu$ as
\[
D_X^cu \ceq D_X^s u \res (\Omega \setminus \J_u).
\]
\end{defi}

\begin{rem}
If $(\R^n,X)$ satisfies both properties $\mathcal R$ and $\mathcal{D}$ (see Definition \ref{def_prod}  below), then the jump part has the representation
 \[
 D^j_X u= \sigma(\cdot, \nu_{\J_u})(u^+-u^-)\nu_{\J_u}\mathcal{S}^{Q-1}\res \J_u.
 \]
 for a suitable function $\sigma\colon\R^n \times \mathbb{S}^{m-1} \to (0,+\infty)$; see~\cite[Theorem 1.7]{dv}. Again,  Heisenberg groups, Carnot groups of step $2$ and Carnot groups of type $\star$ all satisfy both properties, see~\cite[Theorem 4.3]{dv}.
\end{rem}

 \begin{defi}\label{def_prod}
       We say that $(\R^n,X)$ satisfies the \emph{property $\mathcal{D}$} if there exists a function $\sigma:\R^n \times \mathbb{S}^{m-1} \to (0,+\infty)$ such that, for every $C^1_X$-hypersurface $S \se \R^n$ and every $p \in S$, one has
     \[
     \lim_{r \to 0}\frac{\mathcal{S}^{Q-1}(S \cap B(p,r))}{r^{Q-1}}=\sigma(p,\nu_S(p)).
     \]
 \end{defi}

 We mention that the validity of property $\mathcal D$ is related to the broader problem of computing the Federer density for the perimeter measure of surfaces; see, e.g.\  \cite{Magnani15, FSSC15, Magnani17, LecMag21}. If property $\mathcal D$ holds, then the function $\sigma$ of Definition~\ref{def_prod}, is actually explicit in many cases, see e.g.\ \cite{DonMagnani, Magnani22}.

We conclude this section with a couple of technical results that will be useful in the sequel.

\begin{pro}\label{pro_aux1}
Let $u \in \bv_X(\Omega)$.
\begin{enumerate}
    \item[(i)] Let $\xi \in C^\infty_c(\Omega)$. Then $u\xi \in \bv_X(\Omega)$ and
    \[
D_{X_i}(u \xi)=\xi D_{X_i}u+u X_i\xi \leb^n.
\]
\item[(ii)] Let $K \in C^\infty_c(B_E(0,r))$ be spherically symmetric. Then $u*K \in C^\infty(\Omega)$ and for any $ y \in \Omega$ such that $d_E( y, \partial \Omega)>r$ one has
\[
X_i(u*K)(y)=(D_{X_i}u*K)(y)+R_i(u,K; y) 
\]
where 
\begin{equation}\label{eq:R}
R_i(u,K;  y) \ceq \int_\Omega u(x)  \big((\div X_i)(x)K(x- y)-\langle X_i( y)-X_i(x), \nabla K(x- y)\rangle\big)dx.
\end{equation}
\end{enumerate}
\end{pro}
\begin{proof}
Let us first prove $(i)$. Let $\vp \in C^1_c(\Omega)$. We have 
\begin{align*}
-\int_\Omega \vp \xi d(D_{X_i}u)&=\int_\Omega uX_i^*(\vp \xi)d\leb^n=\sum_{t=1}^n\int_\Omega u \frac{\pp (a_{i,t}\vp \xi)}{\pp x_t}d\leb^n=\\
&=\sum_{t=1}^n\int_\Omega u \xi \frac{\pp (a_{i,t} \vp)}{\pp x_j}d\leb^n+\sum_{t=1}^n\int_\Omega u \vp a_{i,t} \frac{\pp \xi}{\pp x_t}d\leb^n=\\
&=\int_\Omega u\xi X_i^* \vp d\leb^n+\int_\Omega u\vp X_i \xi d\leb^n.
\end{align*}
Rearranging the equation we get
\begin{align*}
\int_\Omega u\xi X_i^* \vp d\leb^n=-\int_\Omega \vp d(uX_i \xi \leb^n)-\int_\Omega \vp  d(\xi D_{X_i}u)
\end{align*}
which implies
\[
D_{X_i}(u\xi)=uX_i \xi\leb^n+\xi D_{X_i}u.
\]
For a proof of $(ii)$ see \cite[Lemma 2.6]{garofalonhieu} and \cite{friedrichs}. 
\end{proof}

\begin{lem}[{\cite[Lemma 1.2.1 (i)]{fssc}}]\label{lem_app2}
Let $W_\ve\colon \R^n \to \R$ be, for $\ve>0$, a family of measurable functions supported in $B_E(0,\ve)$, satisfying $|W_\ve(x)|\leq C \ve^{-n}$ for some positive constant $C$ and $\int_{B_E(0,\ve)}W_\ve(x)dx=0$. Then for every $u \in L^1(\Omega)$ we have
\[
\lim_{\ve \to 0}\nl W_\ve * u \nr_{L^1(\Omega)}=0.
\]
\end{lem}

\section{Special functions of Bounded {\it X}-variation}\label{sec_sbvx}
We start by introducing the main object of this paper. From now on, $\Omega\subset\R^n$ is a fixed open set.

\begin{defi}\label{def_sbvx}
Let $u \in \bv_X(\Omega)$. We say that $u$ is a \emph{special function of bounded $X$-variation}, and we write $u \in \sbv_X(\Omega)$, if
\begin{enumerate}
    \item[(i)] $D_X^cu=0$,
    \item[(ii)] $\J_u$ is a countably $X$-rectifiable set.
\end{enumerate}
If $u$ is  in  $\bv_{X,\loc}(\Omega)$ only, we say that $u$ is a \emph{special function of locally bounded $X$-variation}, and we write $u \in \sbv_{X,\loc}(\Omega)$.
\end{defi}

As explained in Remark~\ref{rem_proprietaR}, when $(\R^n,X)$ satisfies property $\mathcal R$ then condition $(ii)$ in Definition~\ref{def_sbvx} is always automatically satisfied by any $u\in \bv_{X,\loc}(\Omega)$.

In the current section the results \cite[Theorems 1.3, 1.4, 1.5, 1.6]{vittone2012} will be crucial: we make the following key observation that will allow us to use the aforementioned theorems.
\begin{rem}
The metric balls in CC spaces, in general, are not $X$-Lipschitz domains (see \cite[Definition 1.1]{vittone2012}) since their boundaries contain characteristic points. However, a $C^1_X$-hypersurface $S$ locally separates an open set into two $X$-regular domains (see \cite[Section 2.4]{vittone2012}) in the following sense: for each $p \in S$, there exists $r>0$ and $f \in C^1_X(B(p,r))$ such that $S \cap B(p,r)= \lbrace f=0 \rbrace$ with $Xf \neq 0$ on $B(p,r)$, thus $B^+\ceq B(p,r) \cap \lbrace f>0 \rbrace$ and $B^- \ceq B(p,r) \cap \lbrace f<0 \rbrace$ satisfy all the properties of $X$-regular domains if we consider the boundary in the relative topology with respect to $B(p,r)$.
\end{rem}

We now provide an equivalent definition for special functions of bounded $X$-variation that will be useful in the sequel.

\begin{pro}\label{prop_defequivsbvx}
    The following statements are equivalent:
    \begin{enumerate}
        \item[(i)] $u\in \sbv_{X,\loc}(\Omega)$;
        \item[(ii)] $u\in \bv_{X,\loc}(\Omega)$ and there exist a countably $X$-rectifiable set $R\subset\Omega$ and a function\footnote{Observe that necessarily $f\in L^1_{\loc}(R,\hau^{Q-1})$, for otherwise $u \not\in \bv_{X,\loc}(\Omega)$.}  $f\colon R\to\R$ such that
        \[
        D_X^su = f\:\nu_R\:\hau^{Q-1}\res R,
        \]
        where $\nu_R $ denotes the horizontal normal to $R$.
    \end{enumerate}
    Moreover, the jump set $\mathcal J_u$ coincides with $R_0:=\{p\in R:f(p)\neq 0\}$ up to $\hau^{Q-1}$-negligible sets.
\end{pro}
\begin{proof}
    Assume $(i)$; then, the jump set $\mathcal J_u$ can be covered, up to a $\hau^{Q-1}$-negligible set, by countably many $C^1_X$-hypersurfaces $(S_j)_{j \in \N}$, that we may assume to be pairwise disjoint. Since $|D_Xu|(S_j\setminus \J_u)=0$ for every $j$, we obtain
    \begin{align*}
        D_X^su &= D_X^ju = D_Xu\res \mathcal J_u =\sum_{j\in\N} D_Xu\res (\J_u\cap S_j).
    \end{align*}
Observe that, locally, each hypersurface $S_j$ separates the space $\R^n$ into two  $X$-regular open sets (see~\cite{vittone2012}): denoting by $P_j^X$ the measure on $S_j$ defined locally as the {\em $X$-perimeter} measure of (each of) these two components, using~\cite[Theorems 1.4 and 1.6]{vittone2012} (see also~\cite[Proposition 3.7]{dv}) one finds
\[
D^s_X u= \sum_{j\in\N} (u^+-u^-) \nu_{S_j}\:P_j^X\res\J_u,
\]
where $u^\pm$ are the {\em traces} (see~\cite{vittone2012}) of $u$ on $S_j$. By~\cite[Theorem 4.2]{AmbrosioAhlfors} we further obtain
\[
D^s_X u= \sum_{j\in\N} (u^+-u^-) \lambda\:\nu_{S_j}\:\hau^{Q-1}\res(\J_u\cap S_j)
\]
    for a suitable $\lambda\colon\bigcup_jS_j\to(0,+\infty)$. Up to changing the sign of 
 $\nu_{\J_u}$, we can write $\nu_{\J_u}=\nu_{S_j}$ $\hau^{Q-1}$-a.e.\ on $\J_u\cap S_j$, hence concluding that
    \[
    D^s_X u= \sum_{j\in\N} (u^+-u^-) \lambda\:\nu_{\J_u}\:\hau^{Q-1}\res(\J_u\cap S_j) ,
    \]
which proves $(ii)$ with $R:=\J_u$ and $f:=(u^+-u^-)\lambda$.

Concerning the opposite implication, assume $(ii)$; clearly, it is not restrictive to assume that $R=R_0$. Cover $R$, up to a $\hau^{Q-1}$-negligible set, by countably many $C^1_X$-hypersurfaces $(S_j)_{j \in \N}$, that we may assume to be pairwise disjoint. Using again~\cite[Theorems 1.4 and 1.6]{vittone2012} (see also~\cite[Proposition 3.7]{dv}) and~\cite[Theorem 4.2]{AmbrosioAhlfors}, for every fixed $j$ we have
\[
D_X^su\res S_j = (u^+-u^-)\lambda\: \nu_{S_j} \hau^{Q-1}\res S_j,
\]
where again $u^\pm$ are the traces of $u$ on $S_j$ and $\lambda\colon S_j\to(0,+\infty)$. On the other hand, by assumption, we have
\[
D_X^su\res S_j = f\:\nu_R\, \hau^{Q-1}\res(R\cap S_j),
\]
which implies that (up to a change of sign for $\nu_{R}$ and $f$)
\begin{align*}
    &u^+-u^-=0\text{ $\hau^{Q-1}$-a.e. on }S_j\setminus R,\\
    &(u^+-u^-)\lambda = f\neq 0 \text{ and }\nu_{S_j}=\nu_R\text{ $\hau^{Q-1}$-a.e. on }S_j\setminus R.
\end{align*}
    Therefore, $\hau^{Q-1}$-a.e.\ point of $R\cap S_j$ is an approximate $X$-jump point and, in particular, $\hau^{Q-1}(R\setminus \J_u)=0$. The proof will be accomplished if we show that $\hau^{Q-1}(\J_u\setminus R)=0$; if, instead, $\hau^{Q-1}(\J_u\setminus R)>0$, then by~\cite[Theorem 1.3 (i)]{dv} we would obtain
    \[
    |D^s_X u|(\J_u\setminus R) = |D_X u|(\J_u\setminus R) >0,
    \]
    which clearly contradicts assumption $(ii)$. This concludes the proof.
\end{proof}

We observe in passing that Proposition \ref{prop_defequivsbvx}, together with the fact that the horizontal normals of two $X$-rectifiable sets $R_1,R_2$ coincide (up to a sign) $\hau^{Q-1}$-a.e. on $R_1 \cap R_2$ (see \cite[Proposition 2.18]{dv}), implies that the space $\sbv_X$ is closed with respect to the usual sum.

The following theorem provides our first main result about $\sbv_X$ functions.

\begin{teo}\label{teo_sbvxchiuso}
The subspace $\sbv_X(\Omega)$ is  closed in $\bv_X(\Omega)$.
    \end{teo}
\begin{proof}
    If $I$ is finite or countable, $u_i \in \sbv_X(\Omega)$ for any $i \in I$ and $\sum_{i \in I}u_i$ converges to $u \in \bv_X(\Omega)$ in the $\bv_X$ norm, then $D_Xu=\sum_{i \in I}D_Xu_i$. Since $\sum_i D^a_{X}u_i$ is absolutely continuous with respect to $\leb^n$ and $\sum_i D^s_X u_i$ is singular, we have 
    \[
    D^a_Xu=\sum_{i \in I}D^a_Xu_i, \qquad D^s_Xu=\sum_{i \in I}D_X^su_i.
    \]
    Proposition~\ref{prop_defequivsbvx} implies that $u\in\sbv_X(\Omega)$, hence $\sbv_X(\Omega)$ is closed in $\bv_X(\Omega)$.
\end{proof}

In the following lemma we denote by $Du=(D_1u,\dots,D_nu)$ the derivatives of $u$ in the sense of distribution; moreover, when $\mu=(\mu_1,\dots,\mu_n)$ is a vector-valued Radon measure and $X(x)=(a_1(x),\dots,a_n(x))$ is a smooth vector field, we denote by $\langle \mu,X\rangle$ the (scalar) Radon measure $\sum_{t=1}^n a_t\mu_t$.

\begin{lem}\label{lem_bvbvx}
The following statements hold:
\begin{enumerate}
   \item[(i)] $\bv_{\loc}(\Omega) \se \bv_{X,\loc}(\Omega)$ and $D_{X_i}u=\langle Du,X_i\rangle$ for every $i=1,\dots,m$.
   \item[(ii)]  $\sbv_{\loc}(\Omega) \se \sbv_{X,\loc}(\Omega)$.
\end{enumerate}
\end{lem}
\begin{proof}
$(i)$  Let $u \in \bv_{\loc}(\Omega)$. For every open set $A \subset \subset \Omega$, for every $1 \leq i \leq m$ and for every $\vp \in C^1_c(A)$ we have
    \begin{align*}
    \int_A u X_i^*\vp d \leb^n&=\sum_{t=1}^n\int_A u \frac{\pp (a_{i,t}\vp)}{\pp x_t}d\leb^n
    =- \int_A \vp\; d\left(\sum_{t=1}^n (a_{i,t}D_tu) \right),
    \end{align*}
as claimed.

$(ii)$ For every $u \in \sbv_{\loc}(\Omega)$ we write $Du= D^{a,E} u \leb^n+(u^+_E-u^-_E)\nu^E\hau^{n-1}_E \res \J^E_u$, where $D^{a,E} u$ denotes the (Euclidean) approximate gradient of $u$,  $\J_u^E$ is the (Euclidean) jump set of $u$ oriented by its (Euclidean) unit normal $\nu^E=(\nu_1^E,\dots \nu_n^E)$ and $u^+_E$ and $u^-_E$ are the (Euclidean) traces of $u$ on $\J_u^E$. By statement $(i)$ we know that $u \in \bv_{X,\loc}(\Omega)$ and 
\begin{align*}
D_{X_i}u=\langle D^{a,E} u, X_i \rangle \leb^n+(u^+_E-u^-_E) \langle \nu^E,X_i \rangle \hau^{n-1}_E \res \J^E_u.
\end{align*}
We know that $ \J_u^E$ is countably rectifiable in the Euclidean sense, hence, up to modifying $ \J_u^E$ on a $\hau^{n-1}_E$-negligible set, there exists a countable collection of $C^1$-hypersurfaces $(S_j)_{j \in \N}$ such that
\[
\J_u^E \subset \bigcup_{j \in \N} S_j.
\]
Without loss of generality we can assume that the $C^1$-hypersurfaces $(S_j)_{j \in \N}$ are pairwise disjoint; in this way $\nu^E$ coincides with the Euclidean unit normal $\nu^E_{S_j}$ to $S_j$ on $\J_u^E\cap S_j$. 
For every $j \in \N$ we introduce the characteristic set $S^\ch_j\subset S_j$ as
\[
S^\ch_j \ceq \lbrace p \in S_j: \operatorname{span}(X_1(p),\dots,X_m(p) ) \se T_pS \rbrace.
\]
For $p \in \J^E_u$ let $k\in\N$ be such that $p \in S_k$; if $p\in S^\ch_k$, then  $\langle X_i(p),\nu^E(p) \rangle=\langle X_i(p),\nu^E_{S_k}(p) \rangle=0$ for every $i \in \lbrace 1, \dots, m \rbrace$, hence we can rewrite $D_{X_i}u$ as
\begin{equation}\label{eq_repdxu}
D_{X_i}u=\langle X_i,D^{a,E} u \rangle \leb^n+(u^+_E-u^-_E) \langle X_i,\nu^E \rangle \hau^{n-1}_E \res R, 
\end{equation}
where
\[
R:= \J^E_u    \setminus \bigcup_{j \in \N} S^\ch_j.
\]
The set $R$ is countably $X$-rectifiable because $R\subset\bigcup_{j \in \N} (S_j\setminus S^\ch_j)$ and each $S_j\setminus S^\ch_j$ is a $C^1_X$-hypersurface. For $p\in R$, let  $k\in\N$ be the unique integer such that $p\in S_k\setminus S^\ch_k$ and let $f$ be a $C^1$ defining function for $S_k$ in a neighborhood of $p$; then
\[
(\nu_R)_i(p)=\frac{X_if(p)}{|Xf(p)|}=\frac{\langle X_i(p),\nabla f(p)\rangle}{|Xf(p)|}
=\frac{\langle X_i(p),\nu^E(p)\rangle}{|Xf(p)|}|\nabla f(p)|,\qquad \forall\: i=1,\dots,m,
\]
i.e.,
\begin{equation}
    \label{eq_normalivarie}
    \langle X_i,\nu^E\rangle = \sigma_1 (\nu_R)_i,\qquad \forall\: i=1,\dots,m
\end{equation}
for a suitable function $\sigma_1\colon R\to(0,+\infty)$.
As in the proof of Proposition~\ref{prop_defequivsbvx} we observe that, locally, each $C^1$- and $C^1_X$-hypersurface $S_j\setminus S_j^\ch$ separates the space into two connected open components of locally finite $X$-perimeter; combining~\cite[Propositions~4.1 and~4.5, Remark~4.7 and Corollary~4.14]{vittone2012} (see also~\cite{DonMagnani}) it can be shown that these $X$-perimeter measures have an integral representation with respect to both $\hau^{Q-1}\res (S_j\setminus S_j^\ch)$  and $\hau^{n-1}_E\res (S_j\setminus S_j^\ch)$. Ultimately, this gives
\begin{equation}
    \label{eq_hausdorffvarie}
    \hau^{n-1}_E \res R = \sigma_2 \hau^{Q-1}\res R
\end{equation}
for a suitable $\sigma_2\colon R\to(0,+\infty)$.
Combining~\eqref{eq_repdxu},~\eqref{eq_normalivarie} and~\eqref{eq_hausdorffvarie} we obtain
\[
D_{X_i}u=\langle X_i,D^{a,E} u \rangle \leb^n+(u^+_E-u^-_E) \sigma_1\sigma_2 (\nu_R)_i \hau^{Q-1} \res R,
\]
i.e.,
\[
D_X^su=(u^+_E-u^-_E) \sigma_1\sigma_2 \nu_R \hau^{Q-1} \res R.
\]
Proposition~\ref{prop_defequivsbvx} implies that $u\in \sbv_{X,\loc}(\Omega)$, as claimed.
\end{proof}

\begin{rem}
    A deeper inspection of the proof of Lemma~\ref{lem_bvbvx} $(ii)$ and, in particular, of the results from~\cite{vittone2012} that were used reveals that, if $u\in\sbv_{\loc}(\Omega)$ and $R$ is as in the proof, then the traces  $u^\pm$ equal the Euclidean ones $u_E^\pm$. 
\end{rem}

The following result is an easy consequence of the celebrated Lusin-type theorem for gradients by~G.~Alberti~\cite{alberti}.

\begin{teo}\label{teo_albertisbvx}
    For every $w \in L^1_{\loc}(\Omega;\R^m)$ there exists $u \in \sbv_{X,\loc}(\Omega)$ such that $D^{\ap}_Xu=w\ \leb^n$-a.e.\ on $\Omega$. 
    
    \noindent Moreover, if $\Omega$ is bounded, then there exists $C=C(\Omega)>0$ such that, for every $w \in L^1(\Omega;\R^m)$, the function $u$ can be chosen in such a way that $|D_Xu|(\Omega) \leq C \nl w \nr_{L^1(\Omega)}$.
\end{teo}
\begin{proof}
Consider the horizontal vector field $X_w:=w_1 X_1+\dots+w_m X_m$. Consider a sequence of open sets $(\Omega_i)_{i \in \N}$ such that, for every $i \in \N$, $\Omega_i \subset\subset \Omega_{i+1}$, $\Omega_i \se \Omega$ and $(\Omega_i)_{i \in \N}$ invades $\Omega$ when $i \to +\infty$. Clearly, $X_w \in L^1(\Omega_i;\R^n)$ for every $i \in \N$.  By~\cite[Theorem 3]{alberti} there exists $u_i\in \sbv(\Omega_i)$ whose (Euclidean) approximate gradient is $X_w|_{\Omega_i}$. For every $i \in \N$ define the sets 
\[
U_i \ceq \begin{cases}
    \Omega_1 \text{ if } i=1\\
    \Omega_{i+1} \setminus \Omega_i  \text{ if } i>1.
\end{cases}
\]
The function $u$ defined as $u_i$ on $U_i$ belongs to $\sbv_{\loc}(\Omega)$ and its (Euclidean) approximate gradient is $X_w$. The latter, together with Lemma~\ref{lem_bvbvx}, proves the first part of the statement.
The second part is a consequence of the estimate stated in~\cite[Theorem 3]{alberti}.
\end{proof}

We conclude with a result, Theorem~\ref{teo_produzionedisbvx}, were we provide a recipe to produce lots of $\sbv_X$ functions: in fact,  {\em any} $L^1$ function on {\em any} countably $X$-rectifiable set can appear as the jump part of an $\sbv_{X}$ function.

\begin{lem}\label{lem_jumpc1xsup}
        Let  $S \se \Omega$ be a $C^1_X$-hypersurface oriented by a normal $\nu$, let $  \theta \in L^1(\hau^{Q-1} \res S)$ and $\delta>0$. Then there exists $u \in \sbv_{X}(\Omega)$ such that
        \[
        D_X^ju \equiv \theta \nu \hau^{Q-1} \res S, \quad \nl u \nr_{L^1(\Omega)}<\delta, \quad \text{ and }\quad |D_Xu| (\Omega) \leq (2+\delta)\nl \theta  \nr_{L^1(\hau^{Q-1} \res S)}.
        \]
\end{lem}
\begin{proof}
Fix  a countable family $(B_j)_{j \in \N}$ of balls, contained in $\Omega$ and with centers on $S$, and functions $f_j \in C^1_X(B_j)$ such that, for every $j \in \N$, 
\[
 S \cap B_j=\lbrace q \in B_j: f_j(q)=0 \rbrace,\qquad Xf_j \neq 0 \text{ on }B_j, \qquad S \subset \bigcup_{j \in \N}B_j. 
\]
We can also assume that $\langle Xf_j,\nu\rangle>0$ on $S \cap B_j$. 
Without loss of generality, we can assume that each ball $B_j$ intersects only a finite number of other balls
 of the collection. Now consider a partition of the unity associated with $(B_j)_{j \in \N}$, i.e., a collection of functions $(\zeta_j)_{j \in \N}$ such that, for every $j \in \N$, 
 \[
 \zeta_j \in C^\infty_c(B_j), \qquad  0 \leq \zeta_j \leq 1,\quad \text{ and }\quad \sum_{j \in \N}\zeta_j \equiv 1 \text{ on }S.
 \]
Fix $j \in \N$, we define 
\begin{equation}\label{eq_defsemi}
B^+_j \ceq \lbrace q \in B_j: f_j(q)>0 \rbrace, \qquad B^-_j \ceq \lbrace q \in B_j: f_j(q)<0 \rbrace.
\end{equation}
For every $j$, let $\theta_j:=\theta/\sigma_j\in L^1(S\cap B_j)$, where $\sigma_j$ is a  function on $S\cap B_j$ with $\inf \sigma_j>0$ that will be introduced later.
Using \cite[Theorem 1.5]{vittone2012} we can find $\tilde{u}_j \in C^\infty(B^+_j) \cap W^{1,1}_X(B^+_j) $\footnote{Remember that $W^{1,1}_X$ is the space of functions $u$ such that both $u$ and $Xu$ belong to $L^1$.} such that spt\;$\tilde u_j\subset\subset B_j$ and  
\[
\nl \tilde{u}_j \nr_{L^1(B^+_j)} \leq \frac{\delta}{2^j}, \qquad \nl X \tilde{u}_j \nr_{L^1(B^+_j)}  \leq \left( 1+\frac{\delta}{2^j} \right) \nl \zeta_j \theta_j \nr_{L^1\big(P^{B^+_j}_X \res S \big)},
\]
and, for $\hau^{Q-1}$-almost every $q \in S \cap B_j$, we have
\[
\lim_{r \to 0} \fint_{B^+_\nu(q,r)} |\tilde{u}_j-\zeta_j \theta_j|d\leb^n=0.
\]
 where $B^+_\nu(q,r)$ is defined as in Definition \ref{def_b+b-}. We define $u_j$ on $\Omega$ as 
 \[
u_j \ceq \begin{cases}
    \tilde{u}_j \text{ on } B^+_j\\
    0 \text{ on } \R^n \setminus B^+_j.
\end{cases}
 \]
 By \cite[Theorem 5.3 and Theorem 1.3]{vittone2012}, $u_j \in \bv_X(\Omega)$ and,  by using the representation of the $X$-perimeter measure and the coarea formula for $\bv_X$ functions (see \cite{AmbrosioAhlfors,fssc})  we can find a constant $C_j>1$ and a $\hau^{Q-1}$-measurable function $\sigma_j: S \cap B_j \to [1/C_j,C_j]$ such that
 \begin{align*}
     D_Xu_j&=(X\tilde{u}_j) \leb^n \res B^+_j+\zeta_j\theta_j \sigma_j \nu \hau^{Q-1}\res (B_j \cap S)\\
     &=(X\tilde{u}_j) \leb^n \res B^+_j+\zeta_j\theta \nu \hau^{Q-1}\res (B_j \cap S),
 \end{align*}
the latter implying that $u_j \in \sbv_X(\Omega)$. The function $ u \ceq \sum_{j \in \N}u_j $ satisfies the statement of the Lemma: clearly $u \in \sbv_X(\Omega)$, $D_X^ju \equiv \theta \nu \hau^{Q-1} \res S$ and $\nl u \nr_{L^1(\Omega)}<\delta $. Now let us prove that the estimate on $|D_Xu|(\Omega)$. We observe
\begin{align*}
|D_Xu |(\Omega)& \leq \sum_{j \in \N}|D_Xu_j|(\Omega) \leq \sum_{j \in \N} \left( \nl X\tilde{u}_j \nr_{L^1(B_j^+)}+\nl \zeta_j \theta \nr_{L^1(\hau^{Q-1} \res S \cap B_j)}  \right)\\
&\leq \sum_{j \in \N}\left(\left( 1+\frac{\delta}{2^j} \right) \nl \zeta_j \theta_j \nr_{L^1\big(P^{B^+_j}_X \res S \big)}+ \nl \zeta_j \theta \nr_{L^1(\hau^{Q-1} \res S\cap B_j)} \right).
\end{align*}
Again, by the representation of the $X$-perimeter measure, we have
\[
\nl \zeta_j \theta_j \nr_{L^1\big(P^{B^+_j}_X \res S \big)}=\nl \zeta_j \theta \nr_{L^1(\hau^{Q-1} \res S \cap B_j)},
\]
the latter implying that
\begin{align*}
|D_Xu|(\Omega)&\leq  \sum_{j \in \N}\left( 2+\frac{\delta}{2^j} \right) \nl \zeta_j \theta \nr_{L^1(\hau^{Q-1} \res S\cap B_j)}\leq (2+\delta)\nl \theta \nr_{L^1(\hau^{Q-1} \res S)},
\end{align*}
concluding the proof.
\end{proof}

\begin{teo}\label{teo_produzionedisbvx}
     Let $S \se \Omega$  be a countably $X$-rectifiable set oriented by $\nu$; let $  \theta \in L^1(\hau^{Q-1} \res S)$ and $\delta>0$ be fixed. Then there exists $u \in \sbv_{X}(\Omega)$ such that
        \[
        D_X^ju \equiv \theta  \nu \hau^{Q-1} \res S, \quad \nl u \nr_{L^1(\Omega)}<\delta, \quad \text{ and }\quad |D_Xu| (\Omega) \leq (2+\delta)\nl \theta  \nr_{L^1(\hau^{Q-1} \res S)}.
        \]
\end{teo}
\begin{proof} 
Since $S$ is countably $X$-rectifiable there exists a countable collection $(S_i)_{i \in \N}$ of $C^1_X$-hyper\-surfaces such that
\[
\hau^{Q-1}\left( S \setminus \bigcup_{i \in \N} S_i \right)=0.
\]
Without loss of generality we can assume that the $C^1_X$-hypersurfaces $(S_i)_{i \in \N}$ are pairwise disjoint. We extend $\theta$ to $0$ outside  $S$ and, for every $i \in \N$, we define
\[
\theta_i \ceq \theta|_{S_i}.
\]
For every $i \in \N$ we use Lemma \ref{lem_jumpc1xsup} to obtain a function $u_i \in \sbv_{X,\loc}(\Omega)$ such that 
  \[
        D_X^ju_i \equiv \theta_i \nu \hau^{Q-1} \res S_i, \quad \nl u_i \nr_{L^1(\Omega)}<\frac{\delta}{2^i}, \quad \text{ and }\quad |D_Xu_i| (\Omega) \leq \left(2+\frac{\delta}{2^i}\right)\nl \theta_i  \nr_{L^1(\hau^{Q-1} \res S_i)}.
        \]
The function $u \ceq \sum_{i \in \N}u_i$ satisfies the statement of the Theorem.
\end{proof}

\section{Proof of Theorem~\ref{teo_introapprox}}\label{sec_approx}
\subsection{Construction of the approximating sequence}
This section is devoted to the construction of the approximating sequence $(u_k)_{k \in \N}$ of  $\sbv_X$ functions that will be used to prove Theorem \ref{teo_introapprox}. In the following Construction~\ref{Con_JeM} we start by proving that, if $u\in\bv_X(\Omega)$ and $\J_u$ is countably $X$-rectifiable, then it is possible to approximate $\J_u$  with a $X$-rectifiable set that can be in turn approximated with a $C_X^1$-hypersurface. We underline that the following construction is valid, as a particular case, for functions $u\in \sbv_X(\Omega)$.
\begin{construction}\label{Con_JeM}
Fix $u \in \bv_X(\Omega)$ with a countably $X$-rectifiable $\mathcal J_u$.
For every $\eta>0$ we define the set
\begin{equation}\label{eq_Jueta}
\J_{u,\eta} \ceq  \left\{x\in \J_u:|u^+(x)-u^-(x)|\geq  
\frac{1}{\eta}  \right\} \cap B(0,\eta).
\end{equation}
 Since the jump set $\J_u$ is countably $X$-rectifiable, also $\J_{u,\eta}$ is countably $X$-rectifiable for every $\eta>0$. Let us moreover observe that the set $\J_{u,\eta}$ is $X$-rectifiable, i.e., that $\hau^{Q-1}(\J_{u,\eta})<\infty$. In fact, thanks to Theorem~\ref{teo_prop3.76}, there exists a positive constant $C>0$, only depending on $\eta$ such that 
\[
\hau^{Q-1}(\J_{u,\eta}) \leq C\eta|D_Xu|(\Omega).
\]
Since the family $\J_{u,\eta}$ is increasing and invades $\J_u$ when $\eta \to +\infty$ we also have
\[
|D_Xu|(\J_u \setminus \J_{u,\eta}) \xrightarrow{\eta \to +\infty}0
\]
so that, for every $k \in \N$, we can choose an $\eta_k>0$ such that $(\eta_k)_{k \in \N}$ is increasing and
\begin{equation}\label{eq_appXrect}
|D_Xu|(\J_u \setminus \J_{u,\eta_k})<\frac{1}{k}.
\end{equation}
For the sake of brevity let us write $\J^k_u \ceq \J_{u,\eta_k}$. Now, for every $\delta>0$, using Lemma \ref{lem_unicaipersup}, we can find a $C^1_X$-hypersurface $M_\delta$ such that 
\[
\hau^{Q-1}(\J_{u}^k \setminus M_\delta)<\delta.
\]
By Theorem~\ref{teo_prop3.76} one has $|D_Xu| \res \J_u \ll \hau^{Q-1} \res \J_u$, so for every $k \in \N$ we can find a $C^1_X$-hypersurface $M_k$ such that
\begin{equation}\label{eq_abscont}
    |D_Xu|(\J_u^k \setminus M_k)<\frac{1}{k}.
\end{equation}
\end{construction}
Before starting the construction of the approximating sequence we need the following Lemma.
\begin{lem}\label{lem_closedset}
Let $u \in \bv_X(\Omega)$ be such that $\J_u$ is countably $X$-rectifiable and consider the function
\begin{align*}
\jf_u: \J_u  &\to \R \times \R \times \mathbb{S}^{m-1}\\
x &\to (u^+(x),u^-(x),\nu_{\J_u}(x)).
\end{align*}
Then for every $k \in \N$ there exist a compact set $C_k \se \J_u^k \cap M_k$ (where $\J_u^k$ and $M_k$ are defined as in Construction \ref{Con_JeM}) and a representative of $\jf_u$ such that $\jf_u|_{C_k}$ is continuous and
\begin{equation}\label{eq_defck}
|D_X u|((\J_u^k \cap M_k) \setminus C_k) <\frac{1}{k}.
\end{equation}
\end{lem}
\begin{proof} 
By \cite[Proposition 2.28]{dv} we can choose a Borel representative of $\jf_u$ so it suffices to use Lusin's Theorem \cite[Theorem 2.3.5]{federer}.
\end{proof}
Using the previous Lemma we can construct the required approximating sequence. Recall that we want to obtain a sequence of functions $(u_k)_{k \in \N}$ such that $u_k \in \sbv_X(\Omega)$, $u_k \in C^\infty(\Omega \setminus \J_{u_k})$ and $u_k$ converges to $u$ in the $\bv_X$ norm. Fix a representative of $u \in \sbv_X(\Omega)$ and $k \in \N$ and consider the compact set $C_k$ given by Lemma \ref{lem_closedset}. For $\ell \in \N$ we define the sets
\begin{align*}
A^1_k& \ceq \left\lbrace x \in \Omega : d_E(x,C_k)>\frac{1}{2} \right\rbrace,\\
A^\ell_k & \ceq \left\lbrace x \in \Omega: \frac{1}{\ell+1}<d_E(x,C_k) <\frac{1}{\ell-1} \right\rbrace \text{ if }\ell>1.
\end{align*}
We observe that
\[
\bigcup_{\ell \in \N}A^\ell_k=\Omega \setminus C_k
\]
and that, for every $\bar \ell \in \N$,  $A^{\bar \ell}_k$ intersects at most two of the sets of the family $(A_k^i)_i$, namely  $A^{\bar \ell+1}_k$ and\footnote{For convenience we also define $A^{\ell}_k \ceq \emptyset$ if $\ell<1$.} $A^{\bar \ell-1}_k$.
Now for $s \in \N$ we define the bounded open sets
\begin{align*}
A^{\ell,1}_k &\ceq  A_k^\ell \cap \lbrace |x|_{\R^n}<2 \rbrace,\\
A^{\ell,s}_k & \ceq  A_k^\ell \cap \lbrace s-1<|x|_{\R^n}<s+1 \rbrace \text{ if }s>1.
\end{align*}
We observe that 
\[
\bigcup_{s \in \N}A\lsk=A^\ell_k
\]
and that, for every $\bar s \in \N$, $A^{\ell,\bar s}_k$ intersects at most two of the sets of the family $(A^{\ell,i}_k)_i$, namely  $A^{\ell,\bar s+1}_k$ and\footnote{For convenience we also define $A\lsk \ceq \emptyset$ if either $\ell<1$ or $s<1$.} $A^{ \ell,\bar s+1}_k$.
Consider a partition of unity on $\Omega \setminus C_k$ associated with $(A\lsk)_{\ell,s \in \N}$, that is,  functions $\xi\lsk \in C^\infty_c(A\lsk)$ such that $0 \leq \xi\lsk \leq 1$ and $\sum_{\ell,s \in \N}\xi\lsk \equiv 1$ on $\Omega \setminus C_k$. Let us also define 
\begin{equation}\label{eq:Z}
Z\lsk \ceq \left\lbrace x \in \R^n: d_E(x, \spt(\xi\lsk)) \leq  \frac{d_E(\pp A\lsk,\spt(\xi\lsk))}{5}\right\rbrace.
\end{equation}
Notice that $Z\lsk$ is compact and $Z\lsk \subset A\lsk$. Fix a mollification kernel, i.e., a spherically symmetric non-negative function $K \in C^\infty_c(B_E(0,1))$ such that $\int_{\R^n}Kd\leb^n=1$. For $\ve>0$ we define $K_\ve(x)=\ve^{-n}K(x/\ve)$. For $k \in \N$ we finally define
\begin{equation}\label{eq_defuk}
u_k \ceq \sum_{\ell,s \in \N}(\xi\lsk u) * K_\eps \text{ on }\Omega \setminus C_k
\end{equation}
where the $\eps$'s are chosen so small that, for every $1 \leq i \leq m$, $1 \leq h \leq n$, $1 \leq t \leq n$, $h\neq t$, we have\footnote{Notice that conditions \eqref{eq_condeps5} and \eqref{eq_condeps6} can be requested because of Lemma \ref{lem_app2}.}

\begin{align}
  &\eps < \frac{1}{2^{\ell +s}k}, \label{eq_condeps1}\\
&\eps<\frac{d_E(\pp A\lsk,\operatorname{spt}(\xi\lsk))}{10},\label{eq_condeps3}\\
&\nl K_{\eps}*(uX_i \xi\lsk)-uX_i \xi\lsk \nr_{L^1(\Omega)}<\frac{1}{2^{\ell +s}k} ,\label{eq_condeps4}\\
&\nl \left(u\xi\lsk\frac{\pp a_{i,t}}{\pp x_t}\right)*W^t_{\eps} \nr_{L^1(\Omega)} < \frac{1}{n2^{\ell +s}k},\label{eq_condeps5}\\
&\nl \left(u \xi\lsk \frac{\pp a_{i,t}}{\pp x_h}\right)*W^{t,h}_{\eps}\nr_{L^1(\Omega)}< \frac{1}{n^22^{\ell +s}k},\label{eq_condeps6}\\
&\nl (\xi\lsk u)*K_{\eps}-\xi\lsk u \nr_{L^1(\Omega)}<\frac{1}{2^{\ell +s}k},\label{eq_condeps7}\\
&\nl (\xi\lsk D_{X_i}^{ap}u)*K_{\eps}-\xi\lsk D^{ap}_{X_i}u\nr_{L^1(\Omega)} < \frac{1}{2^{\ell +s}k},\label{eq_condeps8}\\
&\eps < \frac{1}{100 (\ell+1)(\ell+2)},\label{eq_condeps9}\\
&\eps < \frac{1}{C\nl \nabla K \nr_{L^\infty} \nl  u \nr_{L^1(Z\lsk)} 2^{\ell+s}k}\label{eq_condeps10}.
\end{align}
The number $C>0$ appearing in \eqref{eq_condeps10} is a constant that will be chosen in Proposition \ref{pro_estRS} below, and  $W_{\eps}^{t}$ and $W_{\eps}^{t,h}$ appearing in \eqref{eq_condeps5} and \eqref{eq_condeps6} are defined as 
\begin{align}\label{eq_defW}
W_{\eps}^{t}(x)& \ceq \left( K_{\eps}(x)+x_t\frac{\pp K_{\eps}}{\pp x_t}(x)\right),\\
W_{\eps}^{t,h}(x)& \ceq x_h \frac{\pp K_{\eps}}{\pp x_t}(x),
\end{align} 
being $x_v$ the $v$-th component of $x$, $v=1,\dots, n$.
\begin{rem}\label{rem_pallacontenuta}
Fix $\ell,s,k \in \N$. Then for any $x \in A\lsk$ we have, using condition \eqref{eq_condeps9}, that
\[
B_E(x,\eps) \se \bigcup_{\substack{ \ell-1\leq \alpha \leq  \ell+1\\s-1 \leq \beta \leq s+1 } }A^{\alpha,\beta}_k
\]
Let us also observe that, thanks to condition \eqref{eq_condeps3}, we have that $\spt[(\xi\lsk u) *K_\eps] \se A\lsk$, the latter implying that  the sum in \eqref{eq_defuk} is locally finite, hence $u_k \in C^\infty(\Omega \setminus C_k)$. Moreover, $u_k$ is defined out of a $\leb^n$-negligible set $C_k$ and, from \eqref{eq_condeps7}, $u_k \in L^1(\Omega)$ and $\nl u_k-u \nr_{L^1(\Omega)}\xrightarrow{k \to +\infty}0$. Later, using Lemma \ref{lem_tra}  and Proposition \ref{pro_estRS}, we will prove in Proposition \ref{pro_bvuk} that $u_k \in \sbv_X(\Omega)$.
\end{rem}
\begin{lem}\label{lem_giusti}
Let $u,u_k$ and $C_k$ be defined as in Lemma~\ref{lem_closedset}. Then for every $M>0$ and every $y \in C_k$ one has
\[
\lim_{r \to 0}r^{-M} \int_{B_E(y,r)\cap \Omega }|u_k-u|d\leb^n=0
\]
\end{lem}
\begin{proof}
Fix $r>0$ and $y \in C_k$. From the fact that $\xi\lsk \in C^\infty_c(A\lsk)$ and \eqref{eq_condeps3} we have that
\[
u_k(x)-u(x)=\sum_{s \in \N}\sum_{\ell=\ell_0}^\infty\bigg( (\xi\lsk u) * K_\eps (x)-\xi\lsk u (x) \bigg), \quad \forall x \in B_E(y,r) \cap \Omega 
\]
where $\ell_0 \in \N$ is defined as $\ell_0 \ceq [1/r]$ and $[\cdot]$ denotes the floor function. From \eqref{eq_condeps7} we obtain
\[
\nl u_k-u \nr_{L^1(B_E(y,r)\cap \Omega )}\leq \sum_{s \in \N}\sum_{\ell=\ell_0}^\infty \frac{1}{2^{\ell+s}k}.
\]
From the definition of $\ell_0$ and the fact that for every $M>0$ one has $\lim_{r \to 0^+}r^{-M}2^{-1/r}=0$ we obtain the thesis.
\end{proof}
\begin{lem}\label{lem_giusticc}
Let $u,u_k$ and $C_k$ be defined as in Lemma~\ref{lem_closedset}. Then for every $M>0$ and for every $y \in C_k$ one has
\[
\lim_{r \to 0}r^{-M} \int_{B(y,r)\cap \Omega}|u_k-u|d\leb^n=0
\]
\end{lem}
\begin{proof}
For any $p,q \in \overline{(B_E(y,1))}$ one has $d_E(p,q)\leq Cd(p,q)$ where $C\geq 1$ is a constant only depending on the vector fields $X_i$'s and on a compact set $K\supset\supset(B_E(y,1))$ (see \cite{nsw}).  Hence for any sufficiently small $r>0$, we have $B(y,r)\se B_E(y,Cr)$. The latter implies that for any 
 such $r>0$ we have 
\[
r^{-M} \int_{B(y,r)\cap \Omega}|u_k-u|d\leb^n \leq r^{-M}\int_{B_E(y,Cr)\cap \Omega }|u_k-u|d\leb^n.
\]
The result then follows after letting $r \to 0$ and using Lemma \ref{lem_giusti}.
\end{proof}
\begin{lem}\label{lem_tra}
Let $u,u_k$ and $C_k$ be defined as in Lemma~\ref{lem_closedset}. Then for every $y \in C_k$ 
\[
\lim_{r \to 0}\fint_{B^+_{\nu_{\J_u}(y)}(y,r)}|u_k(x)-u^+(y)|dx=0,\qquad \lim_{r \to 0}\fint_{B^-_{\nu_{\J_u}(y)}(y,r)}|u_k(x)-u^-(y)|dx=0.
\]
\end{lem}
\begin{proof}
We will prove only that $\lim_{r \to 0}\fint_{B^+_{\nu_{\J_u}(y)}(y,r)}|u_k(x)-u^+(y)|dx=0$, the other limit is analogous. For the sake of brevity we write $B^+_r \ceq B^+_{\nu_{\J_u}(y)}(y,r)$ and $u^+\ceq u^+(y)$. Since
\begin{align*}
&\fint_{B^+_r} | u_k(x)-u^+|dx \leq \fint_{B^+_r} | u_k(x)-u(x)|dx+\fint_{B^+_r}|u(x)-u^+|dx,\\
\end{align*}
and $\fint_{B^+_r}|u(x)-u^+|dx \xrightarrow{r \to 0}0$ , it suffices to prove that
\[
\lim_{r \to 0 }\fint_{B^+_r}|u_k(x)-u(x)|dx=0.
\]
We observe that 
\begin{align*}
\fint_{B^+_r}|u_k(x)-u(x)|dx&= \frac{1}{\leb^n(B^+_r)}\int_{B^+_r}|u_k(x)-u(x)|dx \leq \frac{1}{\leb^n(B^+_r)}\int_{B(p,r) \cap \Omega} |u_k(x)-u(x)|dx \\
&= \frac{\leb^n(B(y,r))}{\leb^n(B^+_r)}\frac{1}{\leb^n(B(y,r))}\int_{B(p,r) \cap \Omega }|u_k(x)-u(x)|dx.
\end{align*}
From \cite[Proposition 2.1.5]{dontesi} we have $\frac{\leb^n(B(p,r))}{\leb^n(B^+_r)}\xrightarrow{r \to 0}2$ so the result follows upon letting $r \to 0$, using Lemma \ref{lem_giusticc} and the fact that there exists a positive constant $C>0$ such that $\leb^n(B(y,r)) \geq Cr^Q$ (see for instance \cite[Theorem 1]{nsw}).
\end{proof}

\subsection{Estimates on the total variation}
Fix $k \in \N$, $i \in \lbrace 1,\dots,m\rbrace$ and $y \in \Omega \setminus C_k$. By Proposition \ref{pro_aux1} we have
\begin{align*}
(X_iu_k)(y)&=X_i\left( \sum_{\ell,s \in \N} (\xi\lsk u)*K_{\eps}\right)(y)=\sum_{\ell,s \in \N} X_i[(\xi\lsk u)*K_{\eps}](y)=\\
&=\sum_{\ell,s \in \N} \left[(D_{X_i}(\xi\lsk u)*K_{\eps})(y)+R_i(\xi\lsk u,K_{\eps};y)\right]=\\
&=\sum_{\ell,s \in \N} \bigg\lbrace\bigg[\bigg( uX_i \xi\lsk\leb^n+\xi\lsk D_{X_i}u \bigg)*K_{\eps}\bigg](y)+R_i(\xi\lsk u,K_{\eps};y)\bigg\rbrace=\\
&=\sum_{\ell,s \in \N} \Big[ (\xi \lsk D_{X_i}u)*K_{\eps}(y)+\!\int_{\R^n} \!\!K_{\eps}(y-x)u(x)(X_i\xi\lsk)(x)dx+R_i(\xi\lsk u,K_{\eps};y) \Big].
\end{align*}
For the sake of brevity let us define
\begin{equation}\label{eq:ReS}
S\lski (y) \ceq \int_{\R^n} K_{\eps}(y-x)u(x)(X_i\xi\lsk)(x)dx, \quad R\lski(y) \ceq R_i(\xi\lsk u,K_{\eps};y) 
\end{equation}
so that 
\begin{equation}\label{eq:X_iu_k}
(X_iu_k)(y)=\sum_{\ell,s \in \N} \bigg( (\xi \lsk D_{X_i}u)*K_{\eps}(y)+
    S\lski (y)
+R\lski (y) \bigg).
\end{equation}
Now we want to estimate the $L^1$-norm of the two remainders $R\lski$ and $S\lski$: part of the following Proposition is a rewriting of \cite[Lemma 2.1.1]{fssc} in a language more useful for our purposes.
\begin{pro}\label{pro_estRS}
If $S\lski$ and $R\lski$ are defined as above (see \eqref{eq:R} and \eqref{eq:ReS}), then for any $i=1,\dots, m$
\[
   \nl \sum_{\ell,s \in \N} S\lski \nr_{L^1(\Omega)} \leq \frac{1}{k} \quad \text{ and } \quad  \nl \sum_{\ell,s \in \N} R\lski \nr_{L^1(\Omega)} \leq \frac{3}{k}.
\]
\end{pro}
\begin{proof} We start by estimating $S\lski$. Fix $y\in \Omega \setminus C_k$: since $\sum_{\ell,s \in \N}X_i\xi\lsk \equiv 0$, then
\begin{align*}
\sum_{\ell,s \in \N} S\lski (y) &= \sum_{\ell,s \in \N}K_{\eps}*(uX_i\xi\lsk)(y)\\
&=\sum_{\ell,s \in \N}K_{\eps}*(uX_i\xi\lsk)(y)-\sum_{\ell,s \in \N}(uX_i\xi\lsk)(y)\\
&=\sum_{\ell,s \in \N} \bigg(K_{\eps}*(uX_i\xi\lsk)-uX_i\xi\lsk\bigg)(y).
\end{align*}
Using \eqref{eq_condeps4} and the fact that $\leb^n(C_k)=0$, we immediately obtain $\nl \sum_{\ell,s \in \N} S\lski \nr_{L^1(\Omega)} \leq 1/k$. To estimate $R\lski$ we start by observing that for every $y \in \Omega \setminus C_k$ one has
\begin{align}\label{eq:Respanso}
R\lski(y)&=\int_\Omega \xi\lsk(x)u(x)  \left((\div X_i)(x)K_{\eps}(x-y)-\langle X_i(y)-X_i(x), \nabla K_{\eps}(x-y)\rangle\right)dx \\
&=\sum_{t=1}^n \int_\Omega \xi\lsk(x)u(x) \bigg( \frac{\pp a_{i,t}}{\pp x_t}(x)K_{\eps}(x-y)-(a_{i,t}(y)-a_{i,t}(x))\frac{\pp K_{\eps}}{\pp x_t}(x-y) \bigg)dx
\end{align}
Now, for every $y,x \in Z\lsk$ (see \eqref{eq:Z}) we can expand each $a_{i,t}$ with a Taylor's expansion
\begin{equation}\label{eq_defT}
a_{i,t}(y)-a_{i,t}(x)=\sum_{h=1}^n \frac{\pp a_{i,t}}{\pp x_h}(x)(y-x)_h+T_{i,t}(y,x)
\end{equation}
where $(y-x)_h$ denotes the $h$-component of the vector $y-x\in \mathbb R^n$ and 
\begin{equation}\label{eq_stimaT}
|T_{i,t}(y,x)|\leq C^{k,\ell,s}_{i,t} |y-x|^2
\end{equation}
where $C^{k,\ell,s}_{i,t}$ is a positive constant only depending on the $L^\infty$-norm of the second derivatives of $a_{i,t}$ on $Z\lsk$. Replacing \eqref{eq_defT} into \eqref{eq:Respanso} we obtain:
\begin{align*}
R\lski (y)=&\sum_{t=1}^n \int_\Omega \xi\lsk (x)u(x) \frac{\pp a_{i,t}}{\pp x_t}(x)K_\eps(x-y)dx&
\\&-\sum_{t,h=1}^n \int_\Omega \xi\lsk(x)u(x) \frac{\pp a_{i,t}}{\pp x_h}(x)(y-x)_h\frac{\pp K_\eps}{\pp x_t}(x-y)dx\\
&-\sum_{t=1}^n \int_\Omega \xi\lsk(x)u(x) T_{i,t}(y,x)\frac{\pp K_\eps}{\pp x_t}(x-y)dx\\
=& \sum_{t=1}^n \int_\Omega\xi\lsk(x)u(x)  \frac{\pp a_{i,t}}{\pp x_t}(x) \left( K_\eps(x-y)-(y-x)_t\frac{\pp K_\eps}{\pp x_t}(x-y)\right)dx\\
&- \sum_{t,h=1,t \neq h}^n \int_\Omega\xi\lsk (x)u(x) \frac{\pp a_{i,t}}{\pp x_h}(x)(y-x)_h\frac{\pp K_{\eps}}{\pp x_t}(x-y)dx \\
&-\sum_{t=1}^n \int_\Omega \xi\lsk(x)u(x) T_{i,t}(y,x)\frac{\pp K_\eps}{\pp x_t}(x-y)dx.
\end{align*}
Recall that $K$ (and therefore $K_\eps$) is spherically symmetric so that we can write
\begin{align*}
R\lski (y)=&\sum_{t=1}^n \left(u\xi\lsk\frac{\pp a_{i,t}}{\pp x_t}\right)*W^t_{\eps}(y)+\sum_{t,h=1,t \neq h}^n \left(u \xi\lsk \frac{\pp a_{i,t}}{\pp x_h}\right)*W^{t,h}_{\eps}(y)\\
&-\sum_{t=1}^n \int_\Omega \xi \lsk (x)u(x) T_{i,t}(y,x)\frac{\pp K_{\eps}}{\pp x_t}(x-y)dx
\end{align*}
where $W^t_\eps$ and $W^{t,h}_\eps$ are defined as in \eqref{eq_defW}. Then we have
\begin{align*}
\nl  \sum_{\ell,s \in \N} R\lski \nr_{L^1(\Omega)} \leq& \underbrace{\sum_{\substack{\ell,s \in \N\\ \phantom{t\neq h} }}\sum_{t=1}^n \nl \left(u\xi \lsk \frac{\pp a_{i,t}}{\pp x_t}\right)*W^t_{\eps} \nr_{L^1(\Omega)}}_{(A)} +\underbrace{\sum_{\ell,s \in \N}\:\sum_{\substack{t,h=1\\ t \neq h}}^n\nl \left(u \xi\lsk \frac{\pp a_{i,t}}{\pp x_h}\right)*W^{t,h}_{\eps}\nr_{L^1(\Omega)}}_{(B)}  \\
&+\underbrace{\sum_{\ell,s \in \N}\sum_{t=1}^n \int_{Z\lsk} \int_{B_E(x,\eps)}\left|\xi\lsk(x)u(x) T_{i,t}(y,x)\frac{\pp K_{\eps}}{\pp x_t}(x-y)\right|dydx}_{(C)}
\end{align*}
From \eqref{eq_condeps5} and \eqref{eq_condeps6} we obtain that $(A)\leq 1/k$ and $(B)\leq 1/k$. Concerning $(C)$ we have, using \eqref{eq_stimaT} and the fact that $\left|\frac{\pp K_{\eps}}{\pp x_t}\right|<\nl \nabla K \nr_{L^\infty}{(\eps)}^{-n-1}$ and $|x-y|<\eps$, that
\begin{align*}
(C) &\leq \nl \nabla K \nr_{L^\infty} \sum_{\ell,s \in \N}\sum_{t=1}^n \int_{Z\lsk}\int_{B_E(x,\eps)}{(\eps)}^{-n-1} |\xi\lsk(x)u(x)||T_{i,t}(y,x)|dydx\\
& \leq  \nl \nabla K \nr_{L^\infty} \sum_{\ell,s \in \N}\sum_{t=1}^n \int_{Z\lsk}{(\eps)}^{-n-1} |\xi\lsk(x)u(x)|C^{k,\ell,s}_{i,t}\int_{B_E(x,\eps)}|y-x|^2 dydx\\
& \leq \nl \nabla K \nr_{L^\infty}     \sum_{\ell,s \in \N}\sum_{t=1}^n C^{k,\ell,s}_{i,t} \eps \int_{Z\lsk} |\xi\lsk(x)u(x)|dx\\
\end{align*}
and using \eqref{eq_condeps10} with the specific choice of 
\[
C\coloneqq\sum_{t=1}^nC^{k,\ell,s}_{i,t},
\]
we obtain $(C)\leq 1/k$, concluding the proof.
\end{proof}

\begin{pro}\label{pro_bvuk}
For every $k\in \mathbb N$ the function $u_k$ defined  in \eqref{eq_defuk} satisfies the following properties:
\begin{enumerate}
    \item[(i)] $u_k \in \bv_X(\Omega \setminus C_k)$,
    \item[(ii)] $u_k \in \bv_X(\Omega)$,
    \item[(iii)] $u_k \in \sbv_X(\Omega)$.
\end{enumerate}
\end{pro}
\begin{proof}
\item[(i)] We know from Proposition \ref{pro_aux1} that for every $\ell, s \in \N$ one has $u\lsk \in \bv_X(\Omega \setminus C_k)$ where, for the sake of brevity, we defined $u\lsk \ceq (\xi \lsk u)*K_\eps $. To prove $u_k \in \bv_X(\Omega \setminus C_k)$, it is enough to show that
\begin{equation}\label{eq_aa1}
    \nl \sum_{\ell,s \in \N}u\lsk \nr_{\bv_X(\Omega \setminus C_k)}=\nl \sum_{\ell,s \in \N} u\lsk  \nr_{L^1(\Omega )}+\left|   D_X \left( \sum_{\ell,s \in \N}  u\lsk \right)\right|(\Omega \setminus C_k)<\infty .
\end{equation}
As we mentioned before, thanks to \eqref{eq_condeps7}, $u_k \in L^1(\Omega)$ so that the first addend on the right hand side of \eqref{eq_aa1} is finite. We are left to estimate the second term. Since $u_k \in C^\infty(\Omega \setminus C_k)$ for every $i \in \lbrace 1, \dots, m \rbrace$, we have that
\[
D_{X_i}u_k =X_iu_k \leb^n \text{ on }\Omega \setminus C_k,
\]
meaning that
\begin{align*}
    \left|   D_{X_i} \left( \sum_{\ell,s \in \N}  u\lsk \right)\right|(\Omega \setminus C_k)=\nl  \sum_{\ell, s \in \N} X_iu\lsk \nr_{L^1(\Omega)}.
\end{align*}
Hence, by the decomposition \eqref{eq:X_iu_k}, we write
\[
\nl X_iu_k \nr_{L^1(\Omega)} \leq \nl  \sum_{\ell,s \in \N} (\xi \lsk D_{X_i}u)*K_{\eps}\nr_{L^1(\Omega)}+\nl  \sum_{\ell,s \in \N}
    S\lski \nr_{L^1(\Omega)}+\nl
 \sum_{\ell,s \in \N} R\lski  \nr_{L^1(\Omega)}.
\]
Thanks to Proposition \ref{pro_estRS}, we are only left to prove the boundedness of the first term of the right hand side of the inequality above. Using \cite[Theorem 2.2 (b)]{afp} we have
\begin{align*}
\nl  \sum_{\ell,s \in \N} (\xi \lsk D_{X_i}u)*K_{\eps}\nr_{L^1(\Omega)} \leq \sum_{\ell,s \in \N} \nl (\xi \lsk D_{X_i}u)*K_{\eps}\nr_{L^1(A\lsk)} \leq \sum_{\ell,s \in \N}|D_{X_i}u|(A\lsk+\eps).
\end{align*}
where we have written $A\lsk+\eps$ to denote 
\[
A\lsk+\eps \ceq \bigcup_{x\in A\lsk} B_E(x,\eps).
\]
By Remark \ref{rem_pallacontenuta} we have
\[
A\lsk+\eps \se
\bigcup_{\substack{ \ell-1\leq \alpha \leq  \ell+1\\s-1 \leq \beta \leq s+1 } }A^{\alpha,\beta}_k
\]
so that 
\[
|D_{X_i}u|(A\lsk+\ve\lsk) \leq |D_{X_i}u|
\left(
\bigcup_{\substack{ \ell-1\leq \alpha \leq  \ell+1\\s-1 \leq \beta \leq s+1 } }A^{\alpha,\beta}_k
\right)
\]
and, finally,
\[
\sum_{\ell,s \in \N}|D_{X_i}u|(A\lsk+\eps) \leq 9\sum_{\ell,s \in \N} |D_{X_i}u|(A\lsk)\leq 27\sum_{\ell \in \N} |D_{X_i}u|(A^\ell_k) \leq 81 |D_{X_i}u|(\Omega)
\]
which, together with the fact that $u \in \bv_X(\Omega)$, gives the boundedness of $\nl u_k \nr_{\bv_X(\Omega \setminus C_k)}$. 
\item[(ii)] We aim to prove that $u_k \in \bv_X(\Omega)$. Suppose first that, in addition to our previous assumptions, we have $u \in C^\infty(\Omega)$. Then $C_k=\emptyset$ and clearly $u_k \in C^\infty(\Omega) \subset C^1_X(\Omega)$ for any $k\in \mathbb N$, hence
\begin{align*}
|D_{X_i}u_k|(\Omega)&=\nl X_iu_k \nr_{L^1(\Omega)}\\
&\leq  \nl  \sum_{\ell,s \in \N} (\xi \lsk D_{X_i}u)*K_{\eps}\nr_{L^1(\Omega)}+\nl  \sum_{\ell,s \in \N}
    S\lski \nr_{L^1(\Omega)}+\nl
 \sum_{\ell,s \in \N} R\lski  \nr_{L^1(\Omega)}\\
 &\leq 81|D_{X_i}u|(\Omega)+\frac{4}{k},
\end{align*}
the latter implying that $u_k \in \bv_X(\Omega)$. 

Now we drop the smoothness assumption on $u$ and we just assume that $u \in \bv_X(\Omega)$. We know, by \cite[Theorem 2.2.2]{fssc}, that there exists a sequence $(u^t)_{t \in \N}$ such that, for every $t \in \N$ we have $u^t \in C^\infty(\Omega) \cap \bv_X(\Omega)$ and
\[
\nl u^t-u \nr_{L^1(\Omega)}\xrightarrow{t \to +\infty }0, \quad |D_{X_i}u^t|(\Omega) \xrightarrow{t \to +\infty}|D_{X_i}u|(\Omega),\quad |D_{X_i}u^t|(\Omega) \leq |D_{X_i}u|(\Omega)+\frac{1}{t}
\]
for every $i \in \lbrace 1,\dots,m \rbrace$. Now for every $t \in \N$ consider the approximation sequence $(u_k^t)_{k \in \N}$ constructed as in \eqref{eq_defuk}. For the observations we just made on the approximation of smooth functions we know that $u_k^t \in C^1_X(\Omega) \cap \bv_X(\Omega)$. Let us prove that $\| u_k^t-u_k \|_{L^1(\Omega)} \xrightarrow{t \to +\infty}0$. First we observe that, thanks to Remark \ref{rem_pallacontenuta}, we have $\nl u_k\nr_{L^1(\Omega)} \leq 81\nl u \nr_{L^1(\Omega)}$ so that 
\[
\nl u^t_k-u_k \nr_{L^1(\Omega)} =\nl (u^t-u)_k \nr_{L^1(\Omega)} \leq 81 \nl u^t-u \nr_{L^1(\Omega)}.
\]
The inequality above shows that $\| u_k^t-u_k \|_{L^1(\Omega)} \xrightarrow{t \to +\infty}0$. Then we observe 
\[
|D_{X_i}u^t_k|(\Omega) \leq 81|D_{X_i}u^t|(\Omega)+\frac{4}{k}\leq 81|D_{X_i}u|(\Omega)+\frac{81}{t}+\frac{4}{k}\leq  81|D_{X_i}u|(\Omega)+85.
\]
Passing to the $\liminf$ for $t \to +\infty$ in the above inequality and using the lower semicontinuity of the total variation is enough to obtain $u_k \in \bv_X(\Omega)$.
\item[(iii)] From Lemma \ref{lem_tra}, the construction of $C_k$ in Lemma~\ref{lem_closedset} and the fact that $u_k \in \bv_X(\Omega) \cap C^\infty(\Omega \setminus C_k)$ we obtain that $u_k\in \sbv_X(\Omega)$. 
\end{proof}
The following Lemma will be the last step needed to prove our main result.
\begin{lem}\label{lem_sametrace}
    Let $v,w \in \sbv_X(\Omega)$ and $R \se \J_v \cap \J_w$. Let $ \jf_v$ and $\jf_w$ be defined as in Lemma \ref{lem_closedset} and such that $\jf_v = \jf_w$ $\hau^{Q-1}$-a.e. on $R$. Then
    \[
    |D_X(v-w)|(R)=0.
    \]
\end{lem}
\begin{proof}
    Let us observe that $v-w \in \sbv_X(\Omega)$ and $\hau^{Q-1}(R\cap\mathcal S_{v-w})= 0$. In fact, for $\hau^{Q-1}$-a.e. $p \in R$ one has $\jf_v(p)=\jf_w(p)$ and, letting $\nu \ceq \nu_{\J_v}(p) = \nu_{\J_w}(p)$, we notice that
    \begin{align*}
    \lim_{r \to 0}\fint_{B(p,r)}  |v-w|d\leb^n \leq& \lim_{r \to 0}\fint_{B^+_\nu(p,r)} |v-v^+|d\leb^n+\lim_{r \to 0}\fint_{B^+_\nu(p,r)} |w-w^+|d\leb^n+\\
    &+\lim_{r \to 0}\fint_{B^-_\nu(p,r)} |v-v^-|d\leb^n+\lim_{r \to 0}\fint_{B^-_\nu(p,r)} |w-w^-|d\leb^n=0.
    \end{align*}
    The latter implies that $v-w$ has approximate limit 0 at $p$, i.e., $p\in R \setminus \mathcal S_{v-w}$. This proves that $\hau^{Q-1}(R\cap\mathcal S_{v-w})= 0$ and by  Theorem~\ref{teo_prop3.76}~\eqref{1implication}
    \begin{equation}\label{eq_Riccione1}
    |D_X(v-w)|(R\cap\mathcal S_{v-w})= 0.
    \end{equation}
    Moreover, the measure $\hau^{Q-1}\res(R\setminus\mathcal S_{v-w})$ is $\sigma$-finite and Theorem~\ref{teo_prop3.76}~\eqref{2implication} implies that 
   \begin{equation}\label{eq_Riccione2} 
    |D_X(v-w)|(R\setminus\mathcal S_{v-w})= 0.
    \end{equation}
The desired equality $|D_X(v-w)|(R)= 0$  follows from~\eqref{eq_Riccione1} and~\eqref{eq_Riccione2}.
\end{proof}

\subsection{Proof of Theorem~\ref{teo_introapprox}}
We are ready to prove our main result, Theorem \ref{teo_introapprox}, that we restate for the reader's convenience.
\begin{teo}
Let $u \in \sbv_X(\Omega)$. Then there exists a sequence of functions $(u_k)_{k \in \N} \subset \sbv_X(\Omega)$ and of $C^1_X$-hypersurfaces $(M_k)_{k \in \N} \subset \Omega$ such that, for every $k \in \N$, $\J_{u_k} \se M_k \cap \J_u$, $\J_{u_k}$ is compact, and
\[
\nl u-u_k \nr_{\bv_X(\Omega)}\xrightarrow{k \to +\infty}0, \qquad u_k \in C^\infty(\Omega \setminus \J_{u_k}).
\]
\end{teo}
\begin{proof}

Let $C_k$ and $(u_k)_{k \in \N}$ be defined as in Lemma~\ref{lem_closedset} and \eqref{eq_defuk}. By definition of $\bv_X$-norm we have 
\begin{equation}\label{eq_f1}
    \nl u-u_k \nr_{\bv_X(\Omega)}=\nl u-u_k \nr_{L^1(\Omega)}+|D_X(u-u_k)|(\Omega).
\end{equation}
Thanks to \eqref{eq_condeps7} we  estimate
\begin{equation}\label{eq_f2}
\nl u-u_k \nr_{L^1(\Omega)}<\frac{1}{k}.
\end{equation}
Concerning the other summand, we estimate
\begin{equation}\label{eq_ff2}
\begin{aligned}
&|D_X(u-u_k)|(\Omega)\\ \leq\; & |D_X(u-u_k)|(\Omega \setminus \J_u)+|D_X(u-u_k)|(\J_u \setminus \J^k_u)+|D_X(u-u_k)|(\J_u^k \setminus M_k)\\
&+|D_X(u-u_k)|((\J_u \cap M_k)\setminus C_k)+|D_X(u-u_k)|(C_k).
\end{aligned}
\end{equation}
Because of Lemma \ref{lem_tra}, Proposition \ref{pro_bvuk} and Lemma \ref{lem_sametrace} we have 
\begin{equation}\label{eq_ff1}
|D_X(u-u_k)|(C_k)=0.
\end{equation}
Then, since $u_k \in C^\infty(\Omega \setminus C_k)$, \eqref{eq_appXrect}, \eqref{eq_abscont} and \eqref{eq_defck} imply that
\begin{align}\label{eq_f3}
&|D_X(u-u_k)|((\J_u \cap M_k)\setminus C_k)=|D_Xu|((\J_u \cap M_k)\setminus C_k)<\frac{1}{k},\\
\notag &|D_X(u-u_k)|(\J_u \setminus \J^k_u)=|D_Xu|(\J_u \setminus \J^k_u)<\frac{1}{k},\\
\notag & |D_X(u-u_k)|(\J_u^k \setminus M_k)=|D_Xu|(\J_u^k \setminus M_k)<\frac{1}{k}.
\end{align}
By Theorem \ref{teo_gradapprox}, one has $D_Xu=D^\ap_Xu \leb ^n+D_X^ju$ so that, for every $i \in \lbrace 1,\dots, m \rbrace$, 
\begin{align*}
&|D_{X_i}(u-u_k)|(\Omega \setminus \J_u)=\nl D^{\ap}_{X_i}u-X_iu_k \nr_{L^1(\Omega )}\\
\leq\;&
 \nl D^{\ap}_{X_i}u-\sum_{\ell,s \in \N}[(\xi\lsk D_{X_i}u)*K_{\eps}]\nr_{L^1(\Omega)}+\nl \sum_{\ell,s \in \N} R\lski \nr_{L^1(\Omega )} + \nl \sum_{\ell,s \in \N} S\lski \nr_{L^1(\Omega )} .
\end{align*}
By Proposition \ref{pro_estRS} 
\begin{equation}\label{eq_f4}
\nl \sum_{\ell,s \in \N} R\lski \nr_{L^1(\Omega )}+ \nl \sum_{\ell,s \in \N} S\lski \nr_{L^1(\Omega )} \leq \frac{4}{k},
\end{equation}
while
\begin{align*}
\nl\sum_{\ell,s \in \N}[(\xi\lsk D_{X_i}u)*K_{\eps}]-D^{\ap}_{X_i}u\nr_{L^1(\Omega)}=&\nl\sum_{\ell,s \in \N}[(\xi\lsk (D_{X_i}^{\ap}u\leb^n+D^j_{X_i}u))*K_{\eps}]-D^{\ap}_{X_i}u\nr_{L^1(\Omega)}\\
\leq &\nl \sum_{\ell,s \in \N}[(\xi\lsk D_{X_i}^{\ap}u)*K_{\eps}]-\sum_{\ell,s \in \N}\xi\lsk D^{\ap}_{X_i}u\nr_{L^1(\Omega )}\\
&+\nl\sum_{\ell,s \in \N}(\xi\lsk D^{j}_{X_i}u)*K_{\eps}\nr_{L^1(\Omega )}.
\end{align*}
Thanks to \eqref{eq_condeps8} we have
\begin{equation}\label{eq_f5}
\nl \sum_{\ell,s \in \N}[(\xi\lsk D_{X_i}^{\ap}u)*K_{\eps}]-\sum_{\ell,s \in \N}\xi\lsk D^{\ap}_{X_i}u\nr_{L^1(\Omega )} \leq \frac{1}{k},
\end{equation}
while
\[
\nl\sum_{\ell,s \in \N}(\xi\lsk D^j_{X_i}u)*K_{\eps}\nr_{L^1(\Omega )} \leq \sum_{\ell,s \in \N } \nl (\xi\lsk D^j_{X_i}u)*K_{\eps}\nr_{L^1(A\lsk)}.
\]
Fix $\ell,s \in \N$. By \cite[Theorem 2.2 (b)]{afp} we can write
\[
\nl(\xi\lsk D^j_{X_i}u)*K_{\eps}\nr_{L^1(A\lsk)} \leq  |D^j_{X_i}u|(A\lsk+\ve\lsk),
\]
which in turn, by Remark \ref{rem_pallacontenuta}, satisfies the following inclusion
\[
A\lsk+\eps \se
\bigcup_{\substack{ \ell-1\leq \alpha \leq  \ell+1\\s-1 \leq \beta \leq s+1 } }A^{\alpha,\beta}_k
\]
Hence
\[
|D^j_{X_i}u|(A\lsk+\ve\lsk) \leq |D^j_{X_i}u|
\left(
\bigcup_{\substack{ \ell-1\leq \alpha \leq  \ell+1\\s-1 \leq \beta \leq s+1 } }A^{\alpha,\beta}_k
\right)
\]
Finally, we obtain
\begin{equation}\label{eq_f6}
\begin{split}
\sum_{\ell,s \in \N}\nl (\xi\lsk D^j_{X_i}u)*K_{\eps}\nr_{L^1(A\lsk)}&\leq  \sum_{\ell,s \in \N}|D^j_{X_i}u|
\left(
\bigcup_{\substack{ \ell-1\leq \alpha \leq  \ell+1\\s-1 \leq \beta \leq s+1 } }A^{\alpha,\beta}_k
\right)  \leq 81|D^j_{X_i}u|(\Omega \setminus C_k) \leq \frac{81}{k}.
\end{split}
\end{equation}
Combining \eqref{eq_f2}, \eqref{eq_ff2}, \eqref{eq_ff1}, \eqref{eq_f3}, \eqref{eq_f4}, \eqref{eq_f5}, \eqref{eq_f6} with \eqref{eq_f1} and letting $k \to +\infty$ one achieves the desired conclusion.
\end{proof}


\begin{thebibliography}{10}


\bibitem{abb}
{\sc Agrachev, A., Barilari, D., and Boscain, U.}
\newblock {\em A comprehensive introduction to sub-{R}iemannian geometry}, vol.~181 of {\em Cambridge Studies in Advanced Mathematics}.
\newblock Cambridge University Press, Cambridge, 2020.
\newblock From the Hamiltonian viewpoint, With an appendix by Igor Zelenko.

\bibitem{alberti}
{\sc Alberti, G.}
\newblock A {L}usin type theorem for gradients.
\newblock {\em J. Funct. Anal. 100}, 1 (1991), 110--118.

\bibitem{AmbrosioAhlfors}
{\sc Ambrosio, L.}
\newblock Some fine properties of sets of finite perimeter in {A}hlfors regular metric measure spaces.
\newblock {\em Adv. Math. 159}, 1 (2001), 51--67.

\bibitem{abbf16}
{\sc Ambrosio, L., Bourgain, J., Brezis, H., and Figalli, A.}
\newblock B{MO}-type norms related to the perimeter of sets.
\newblock {\em Comm. Pure Appl. Math. 69}, 6 (2016), 1062--1086.

\bibitem{afp}
{\sc Ambrosio, L., Fusco, N., and Pallara, D.}
\newblock {\em Functions of bounded variation and free discontinuity problems}.
\newblock Oxford Mathematical Monographs. The Clarendon Press, Oxford University Press, New York, 2000.

\bibitem{agm15}
{\sc Ambrosio, L., Ghezzi, R., and Magnani, V.}
\newblock B{V} functions and sets of finite perimeter in sub-{R}iemannian manifolds.
\newblock {\em Ann. Inst. H. Poincar\'e{} C Anal. Non Lin\'eaire 32}, 3 (2015), 489--517.

\bibitem{am03}
{\sc Ambrosio, L., and Magnani, V.}
\newblock Weak differentiability of {BV} functions on stratified groups.
\newblock {\em Math. Z. 245}, 1 (2003), 123--153.

\bibitem{AmbrosioMirandaPallaraSBV}
{\sc Ambrosio, L., Miranda, Jr., M., and Pallara, D.}
\newblock Special functions of bounded variation in doubling metric measure spaces.
\newblock In {\em Calculus of variations: topics from the mathematical heritage of {E}.\ {D}e {G}iorgi}, vol.~14 of {\em Quad. Mat.} Dept. Math., Seconda Univ. Napoli, Caserta, 2004, pp.~1--45.

\bibitem{as10}
{\sc Ambrosio, L., and Scienza, M.}
\newblock Locality of the perimeter in {C}arnot groups and chain rule.
\newblock {\em Ann. Mat. Pura Appl. (4) 189}, 4 (2010), 661--678.

\bibitem{ag78}
{\sc Anzellotti, G., and Giaquinta, M.}
\newblock B{V} functions and traces.
\newblock {\em Rend. Sem. Mat. Univ. Padova 60\/} (1978), 1--21.

\bibitem{bu95}
{\sc Biroli, M., and Mosco, U.}
\newblock Sobolev and isoperimetric inequalities for {D}irichlet forms on homogeneous spaces.
\newblock {\em Atti Accad. Naz. Lincei Cl. Sci. Fis. Mat. Natur. Rend. Lincei (9) Mat. Appl. 6}, 1 (1995), 37--44.

\bibitem{bbm}
{\sc Bourgain, J., Brezis, H., and Mironescu, P.}
\newblock A new function space and application.
\newblock {\em J. Eur. Math. Soc. (JEMS) 17}, 9 (2015), 2083--2101.

\bibitem{bmp12}
{\sc Bramanti, M., Miranda, Jr., M., and Pallara, D.}
\newblock Two characterization of {BV} functions on {C}arnot groups via the heat semigroup.
\newblock {\em Int. Math. Res. Not. IMRN}, 17 (2012), 3845--3876.

\bibitem{cdg}
{\sc Capogna, L., Danielli, D., and Garofalo, N.}
\newblock The geometric {S}obolev embedding for vector fields and the isoperimetric inequality.
\newblock {\em Comm. Anal. Geom. 2}, 2 (1994), 203--215.

\bibitem{CapGarAhlfors}
{\sc Capogna, L., and Garofalo, N.}
\newblock Ahlfors type estimates for perimeter measures in {C}arnot-{C}arath\'eodory spaces.
\newblock {\em J. Geom. Anal. 16}, 3 (2006), 455--497.

\bibitem{cm20}
{\sc Comi, G.~E., and Magnani, V.}
\newblock The {G}auss-{G}reen theorem in stratified groups.
\newblock {\em Adv. Math. 360\/} (2020), 106916, 85.

\bibitem{dgn98}
{\sc Danielli, D., Garofalo, N., and Nhieu, D.-M.}
\newblock Trace inequalities for {C}arnot-{C}arath\'eodory spaces and applications.
\newblock {\em Ann. Scuola Norm. Sup. Pisa Cl. Sci. (4) 27}, 2 (1998), 195--252.

\bibitem{DanGarNhi}
{\sc Danielli, D., Garofalo, N., and Nhieu, D.-M.}
\newblock Non-doubling {A}hlfors measures, perimeter measures, and the characterization of the trace spaces of {S}obolev functions in {C}arnot-{C}arath\'eodory spaces.
\newblock {\em Mem. Amer. Math. Soc. 182}, 857 (2006), x+119.

\bibitem{degiorgiambrosio}
{\sc De~Giorgi, E., and Ambrosio, L.}
\newblock New functionals in the calculus of variations.
\newblock {\em Atti Accad. Naz. Lincei Rend. Cl. Sci. Fis. Mat. Nat. (8) 82}, 2 (1988), 199--210.

\bibitem{dpfp}
{\sc De~Philippis, G., Fusco, N., and Pratelli, A.}
\newblock On the approximation of {SBV} functions.
\newblock {\em Atti Accad. Naz. Lincei Rend. Lincei Mat. Appl. 28}, 2 (2017), 369--413.

\bibitem{stokes}
{\sc Di~Marco, M., Julia, A., Nicolussi~Golo, S., and Vittone, D.}
\newblock {S}ubmanifolds with boundary and {S}tokes' {T}heorem in {H}eisenberg groups.
\newblock {\em Trans. Amer. Math. Soc. 378}, 7 (2025), 4955--4990.

\bibitem{dontesi}
{\sc Don, S.}
\newblock {F}unctions of bounded variation in {C}arnot-{C}arathéodory spaces, 2019.
\newblock PhD Thesis available online at \url{https://dottorato.math.unipd.it/sites/default/files/tesi_definitiva_Sebastiano_Don.pdf}.

\bibitem{DonMagnani}
{\sc Don, S., and Magnani, V.}
\newblock Surface measure on, and the local geometry of, sub-{R}iemannian manifolds.
\newblock {\em Calc. Var. Partial Differential Equations 62}, 9 (2023), Paper No. 254, 42.

\bibitem{dmv19}
{\sc Don, S., Massaccesi, A., and Vittone, D.}
\newblock Rank-one theorem and subgraphs of {BV} functions in {C}arnot groups.
\newblock {\em J. Funct. Anal. 276}, 3 (2019), 687--715.

\bibitem{dv19}
{\sc Don, S., and Vittone, D.}
\newblock A compactness result for {BV} functions in metric spaces.
\newblock {\em Ann. Acad. Sci. Fenn. Math. 44}, 1 (2019), 329--339.

\bibitem{dv}
{\sc Don, S., and Vittone, D.}
\newblock Fine properties of functions with bounded variation in {C}arnot-{C}arath\'{e}odory spaces.
\newblock {\em J. Math. Anal. Appl. 479}, 1 (2019), 482--530.

\bibitem{federer}
{\sc Federer, H.}
\newblock {\em Geometric measure theory}.
\newblock Die Grundlehren der mathematischen Wissenschaften, Band 153. Springer-Verlag New York, Inc., New York, 1969.

\bibitem{fgw94}
{\sc Franchi, B., Gallot, S., and Wheeden, R.~L.}
\newblock Sobolev and isoperimetric inequalities for degenerate metrics.
\newblock {\em Math. Ann. 300}, 4 (1994), 557--571.

\bibitem{fssc}
{\sc Franchi, B., Serapioni, R., and Serra~Cassano, F.}
\newblock Meyers-{S}errin type theorems and relaxation of variational integrals depending on vector fields.
\newblock {\em Houston J. Math. 22}, 4 (1996), 859--890.

\bibitem{fssc01}
{\sc Franchi, B., Serapioni, R., and Serra~Cassano, F.}
\newblock Rectifiability and perimeter in the {H}eisenberg group.
\newblock {\em Math. Ann. 321}, 3 (2001), 479--531.

\bibitem{fssc03}
{\sc Franchi, B., Serapioni, R., and Serra~Cassano, F.}
\newblock On the structure of finite perimeter sets in step 2 {C}arnot groups.
\newblock {\em J. Geom. Anal. 13}, 3 (2003), 421--466.

\bibitem{FSSC15}
{\sc Franchi, B., Serapioni, R.~P., and Serra~Cassano, F.}
\newblock Area formula for centered {H}ausdorff measures in metric spaces.
\newblock {\em Nonlinear Anal. 126\/} (2015), 218--233.

\bibitem{friedrichs}
{\sc Friedrichs, K.~O.}
\newblock The identity of weak and strong extensions of differential operators.
\newblock {\em Trans. Amer. Math. Soc. 55\/} (1944), 132--151.

\bibitem{fms16}
{\sc Fusco, N., Moscariello, G., and Sbordone, C.}
\newblock A formula for the total variation of {$SBV$} functions.
\newblock {\em J. Funct. Anal. 270}, 1 (2016), 419--446.

\bibitem{gn96}
{\sc Garofalo, N., and Nhieu, D.-M.}
\newblock Isoperimetric and {S}obolev inequalities for {C}arnot-{C}arath\'eodory spaces and the existence of minimal surfaces.
\newblock {\em Comm. Pure Appl. Math. 49}, 10 (1996), 1081--1144.

\bibitem{garofalonhieu}
{\sc Garofalo, N., and Nhieu, D.-M.}
\newblock Lipschitz continuity, global smooth approximations and extension theorems for {S}obolev functions in {C}arnot-{C}arath\'{e}odory spaces.
\newblock {\em J. Anal. Math. 74\/} (1998), 67--97.

\bibitem{JNGVIMRN}
{\sc Julia, A., Nicolussi~Golo, S., and Vittone, D.}
\newblock Lipschitz functions on submanifolds of {H}eisenberg groups.
\newblock {\em Int. Math. Res. Not. IMRN}, 9 (2023), 7399--7422.

\bibitem{Lahti2020}
{\sc Lahti, P.}
\newblock Approximation of {BV} by {SBV} functions in metric spaces.
\newblock {\em Journal of Functional Analysis 279}, 11 (2020), 108763.

\bibitem{LecMag21}
{\sc Leccese, G., and Magnani, V.}
\newblock A study of measure-theoretic area formulas.
\newblock {\em Ann. Mat. Pura Appl. (4)\/} (2021), online.

\bibitem{magnani02}
{\sc Magnani, V.}
\newblock {\em Elements of geometric measure theory on sub-{R}iemannian groups}.
\newblock Scuola Normale Superiore, Pisa, 2002.

\bibitem{Magnani15}
{\sc Magnani, V.}
\newblock On a measure-theoretic area formula.
\newblock {\em Proc. Roy. Soc. Edinburgh Sect. A 145\/} (2015), 885--891.

\bibitem{Magnani17}
{\sc Magnani, V.}
\newblock A new differentiation, shape of the unit ball, and perimeter measure.
\newblock {\em Indiana Univ. Math. J. 66}, 1 (2017), 183--204.

\bibitem{Magnani22}
{\sc Magnani, V.}
\newblock Rotational symmetries and spherical measure in homogeneous groups.
\newblock {\em J. Geom. Anal. 32}, 4 (2022), Paper No. 119, 31.

\bibitem{marchi14}
{\sc Marchi, M.}
\newblock Regularity of sets with constant intrinsic normal in a class of {C}arnot groups.
\newblock {\em Ann. Inst. Fourier (Grenoble) 64}, 2 (2014), 429--455.

\bibitem{nsw}
{\sc Nagel, A., Stein, E.~M., and Wainger, S.}
\newblock Balls and metrics defined by vector fields. {I}. {B}asic properties.
\newblock {\em Acta Math. 155}, 1-2 (1985), 103--147.

\bibitem{Selby}
{\sc Selby, C.}
\newblock An extension and trace theorem for functions of {H}-bounded variation in {C}arnot groups of step 2.
\newblock {\em Houston J. Math. 33}, 2 (2007), 593--616.

\bibitem{sy03}
{\sc Song, Y.~Q., and Yang, X.~P.}
\newblock {$BV$} functions in the {H}eisenberg group {$H^n$}.
\newblock {\em Chinese Ann. Math. Ser. A 24}, 5 (2003), 541--554.

\bibitem{vittone2012}
{\sc Vittone, D.}
\newblock Lipschitz surfaces, perimeter and trace theorems for {BV} functions in {C}arnot-{C}arath\'{e}odory spaces.
\newblock {\em Ann. Sc. Norm. Super. Pisa Cl. Sci. (5) 11}, 4 (2012), 939--998.

\bibitem{vittone22}
{\sc Vittone, D.}
\newblock Lipschitz graphs and currents in {H}eisenberg groups.
\newblock {\em Forum Math. Sigma 10\/} (2022), Paper No. e6, 104.

\end{thebibliography}
\end{document}